\documentclass[11pt,a4paper]{article}
\usepackage{mathrsfs}
\usepackage{amsmath}
\usepackage{amssymb,amsfonts,amsthm}
\usepackage{amsfonts}
\usepackage{graphicx,subfigure,epstopdf}
\usepackage[usenames]{color}
\usepackage{url}
\usepackage{algorithm,algorithmic}
\usepackage{dsfont,verbatim}
\usepackage[colorlinks,linktocpage,linkcolor=blue]{hyperref}
\usepackage[margin=1.1in]{geometry}
\usepackage{multirow}
\usepackage{calc}
\usepackage{diagbox}
\usepackage{enumerate,enumitem}
\usepackage[normalem]{ulem}
\usepackage{cite}
\usepackage{booktabs}
\usepackage{bm}
\allowdisplaybreaks

\definecolor{darkred}{rgb}{.7,0,0}

\definecolor{green}{rgb}{0,0.7,0}

\def \endproof{\vrule height8pt width 5pt depth 0pt}

\newtheoremstyle{thmm}{1.5ex plus 1ex minus .2ex}{1.5ex plus 1ex minus.2ex}{\rmfamily}{}{\bfseries}{}{1em}{} \theoremstyle{thmm}
\newtheorem{theorem}{Theorem}[section]
\newtheorem{assumption}{Assumption}[section]
\newtheorem{lemma}{Lemma}[section]
\newtheorem{definition}{Definition}[section]
\newtheorem{proposition}{Proposition}[section]

\newtheorem{exercise}{Exercise}[section]

\newtheorem{remark}{Remark}[section]
\renewenvironment{proof}[1][Proof]{\noindent\textit{#1. }
}{\hfill$\square$}

\renewcommand{\theequation}{\thesection.\arabic{equation}}

\newcommand{\vertiii}[1]
{{\left\vert\kern-0.25ex\left
\vert\kern-0.25ex\left\vert #1
    \right\vert\kern-0.25ex\right
\vert\kern-0.25ex\right\vert}}

\def\d{{\mathrm d}}

\def\R {{\mathbb R}}

\def\R{\mathbb{R}}
\def\O{\mathcal{O}}

\def\Z{\mathbb{Z}}

\def\d{\mathrm{d}}
\def\eps{\varepsilon}

\def\Z{{\mathbb Z}}

\def\le{\leqslant}
\def\eps{\varepsilon}
\def\ge{\geqslant}

\def\Omega{\varOmega}
\def\Om{\varOmega}
\def\Delta{\varDelta}

\def\bex{\begin{exercise}\upshape}
\def\eex{\end{exercise}}
\def\be{\begin{equation}}
\def\en{\end{equation}}
\def\ben{\begin{equation*}}
\def\enn{\end{equation*}}

\usepackage{fancyhdr}

\title{\Large
Improved error estimates for a modified exponential Euler method\\
for the semilinear stochastic heat equation with rough initial data\thanks{This work is supported in part by the National Natural Science Foundation of China (NSFC grants 12071020, 12131005, U2230402), the Research Grants Council of Hong Kong (GRF Project No. PolyU15300519), and an internal grant of The Hong Kong Polytechnic University (Project ID: P0038843, Work Programme: ZVX7).}
\\[10pt]
}
\author{Xinping Gui\thanks{Beijing Computational Science Research Center, Beijing 100193, China. 
E-mail address: gui@csrc.ac.cn}
,
\quad
Buyang Li\thanks{Department of Applied Mathematics, The Hong Kong Polytechnic University, Hung Hom, Hong Kong.
Email address: buyang.li@polyu.edu.hk}
,
\quad%\mbox{and}\quad 
Jilu Wang\thanks{School of Science, Harbin Institute of Technology, Shenzhen 518055, China. E-mail address: wangjilu@hit.edu.cn}
}

\date{}

\begin{document}

\maketitle

\thispagestyle{fancy}
\fancyhead[RO,LE]{Accepted for publication by Science China Mathematics}
\vspace{-10pt}

\begin{abstract}
A class of stochastic Besov spaces $B^p L^2(\Omega;\dot H^\alpha(\mathcal{O}))$, $1\le p\le\infty$ and $\alpha\in[-2,2]$, is introduced to characterize the regularity of the noise in the semilinear stochastic heat equation 
\begin{equation*}
\d u -\Delta u \d t =f(u) \d t +\d W(t) ,
\end{equation*}
%
%The numerical solution of the problem is considered 
%in a bounded domain under the homogeneous Dirichlet boundary condition, driven by a class of additive noises 
under the following conditions for some $\alpha\in(0,1]$: 
$$
\Big\| \int_0^te^{-(t-s)A}\d W(s) \Big\|_{L^2(\Omega;L^2(\mathcal{O}))} \le C t^{\frac{\alpha}{2}}
\quad\mbox{and}\quad
\Big\| \int_0^te^{-(t-s)A}\d W(s) \Big\|_{B^\infty L^2(\Omega;\dot H^\alpha(\mathcal{O}))}\le C.
$$
The conditions above are shown to be satisfied by both trace-class noises (with $\alpha=1$) and one-dimensional space-time white noises (with $\alpha=\frac12$). The latter would fail to satisfy the conditions with $\alpha=\frac12$ if the stochastic Besov norm $\|\cdot\|_{B^\infty L^2(\Omega;\dot H^\alpha(\mathcal{O}))}$ is replaced by the classical Sobolev norm $\|\cdot\|_{L^2(\Omega;\dot H^\alpha(\mathcal{O}))}$, and this often causes reduction of the convergence order in the numerical analysis of the semilinear stochastic heat equation. 
%For a Lipschitz continuous nonlinear function $f(u)$, 
In this article, the convergence of a modified exponential Euler method, with a spectral method for spatial discretization, is proved to have order $\alpha$ in both time and space 
% , with error bound of 
%$O(\tau^{\alpha}+ M^{-\alpha}t_n^{-(\alpha+\beta)/2})$ 
for possibly nonsmooth initial data in $L^4(\Omega;\dot{H}^{\beta}(\O))$ with $\beta>-1$, 
% 这里要不要把 \Omega 上的正则性加上？有点不确定怎么加
%{\color{red}$L^p(\Omega;\dot{H}^{\beta}(\O))$ with any $\beta>-1$
%and all $p\in [2,4]$,}
by utilizing the real interpolation properties of the stochastic Besov spaces and a class of locally refined  stepsizes to resolve the singularity of the solution at $t=0$.  \\
% and an auxiliary temporal semi-discrete scheme.
%Numerical experiments are also provided to support the theoretical analysis. \\ 

%{\b stochastic Besov space, semilinear stochastic heat equation, space-time white noise, exponential Euler method, spectral method, real interpolation}

\end{abstract}

%\blfootnote{This work is partially supported by a Hong Kong RGC grant (project no. 15301818). }

%\vspace{-10pt}

\setlength\abovedisplayskip{4pt}
\setlength\belowdisplayskip{4pt}

\section{Introduction}\label{Se:1}
We consider the numerical approximation to the semilinear stochastic heat equation
\begin{equation}
\label{SPDE1}
\left \{
\begin{aligned}
\d u -\Delta u \d t&=f(u) \d t +\d W(t) \quad\mbox{for}\,\,\, t\in(0,T] ,\\
u(0)&=u^0 , 
\end{aligned}
\right .
\end{equation}
in a convex polygonal domain $\O\subset\R^d$, $d\ge 1$, up to a given time $T>0$,  
with a given nonlinear drift
%{\textcolor{blue}{mapping $f:L^2(\O)\to L^2(\O)$} }
function $f:\R\rightarrow\R$ 
and a given initial value $u^0$, 
where $\Delta:H^1_0(\O)\cap H^2(\O)\rightarrow L^2(\O)$ is the Dirichlet Laplacian operator, and 
%Here $\O\in \mathbb{R}^d$ $(d=1,2,3)$ is 
%chosen so that there exists  a completely orthogonal basis $\{\phi_k\}_{k=1}^{\infty}$ for the space $L^2(\mathcal{O})$. 
$W(t)$ is an $L^2(\mathcal{O})$-valued $Q$-Wiener process on a probability space $(\Om,\mathcal{F},\mathbb{P})$ with a normal filtration $\{\mathcal{F}_t\}_{t\ge 0}$. In particular, $W(t)$ has the following expression: 
\begin{align}
W(t) 
=& \sum_{k=1}^\infty \sqrt{\mu_k} \phi_k W_k(t),
\end{align}
where $W_k(t)$, $k=1, 2, \dots$, are real-valued independent Brownian motions, and $\phi_k$, $k=1, 2, \dots$, are common eigenfunctions of the operators $Q$ and $A=-\Delta$, i.e., 
 $$
Q\phi_k=\mu_k\phi_k\quad \mbox{and}\quad A\phi_k=\lambda_k\phi_k \quad\mbox{for}\,\,\, k=1,2,\dots
$$
%for $k\in \mathbb{N^+}$.
%It is well-known that $\{\phi_k\}_{k=1}^{\infty}$ is a completely orthogonal basis of the space $L^2(\mathcal{O})$ 
It is well known that ${\rm tr}(Q)=\sum_{k=1}^{\infty}\mu_k<\infty$ decides a genuine Wiener process which determines a trace-class noise, and $Q=I$ gives a cylindrical Wiener process which determines a space-time white noise; see Prato \& Zabczyk \cite{Da-Prato1992} for more details. 
The trace-class noise is much smoother than the space-time white noise, and therefore the analysis for the latter is often more challenging. 
%Hence we intend to provide some analysis in view of these two typical cases. 

When the initial value is sufficiently smooth, 
it is well known that problem \eqref{SPDE1} has a unique mild solution satisfying the integral equation 
%A function $u\in L^1(0,T;L^2(\Omega;L^2(\O))) \cap C((0,T];L^2(\Omega;L^2(\O)))$ satisfying the integral formulation 
\be\label{mild-solution}
u(t)=e^{-tA}u^0+\int_0^t e^{-(t-s)A}f(u(s))\d s+\int_0^t e^{-(t-s)A}\d W(s) ,
\quad 
\en
in $L^2(\Omega;L^2(\O))$ for all $t\in(0,T]$, where $\{e^{-tA}:t\ge 0\}$ is the analytic semigroup generated by the operator $A$; see \cite{Kruse-2014}. Based on the expression in \eqref{mild-solution}, Jentzen \& Kloeden \cite{Jentzen-Kloeden-2009} proposed the exponential Euler method for semilinear stochastic problems, with the spectral Galerkin method in space. 
%The two methods are both easy to implement and highly efficient. Therefore
%a lot of efforts have been put into analyzing this fully discrete scheme. 
%As a matter of fact,
%this is also almost the scheme we apply in this paper, to be specified,
%\begin{align}
%U_M^1
%&=
%e^{-\tau_1A}U_M^0\quad\quad \mbox{with}\;\;U_M^0 = P_M u^0,\label{fully1}\\
%U_M^n 
%&=
%e^{-\tau_n A}U_M^{n-1}+\int_{t_{n-1}}^{t_n}e^{-(t_n-s)A}P_Mf(U_M^{n-1})ds+\int_{t_{n-1}}^{t_n}e^{-(t_n-s)A}P_MdW(s) \quad \label{fully2}
%\end{align}
%for  $n\ge 2$,
%where definitions of the time partition and the projection operator $P_M$ are given in the beginning of Sections \ref{time-semi} and \ref{space-semi} 
%separately.  
For an abstract semilinear stochastic equation in a Hilbert space $H$, 
under the assumptions 
%However, with strict restrictions imposed
%on the drift term as follows,
\begin{align*}
&|f'(x)-f'(y)|
\le 
C|x-y|,
\quad 
|A^{-r}f'(x)A^rv| 
\le
C |v|, \\
&|A^{-1}f''(x)(v,w)|
\le
C|A^{-\frac12}v| |A^{-\frac12}w|,
\end{align*}
where $x, y, v, w\in H$ and $r\in\{0,\frac12,1\}$, 
Jentzen \& Kloeden \cite{Jentzen-Kloeden-2009} proved the strong convergence with an error bound of  
$O(\tau \ln(\frac{1}{\tau})+M^{-\frac12+\varepsilon})$ for one-dimensional space-time white noises, where $\varepsilon$ can be an arbitrary small number. These assumptions exclude nonlinear Nemytskii operators and therefore cannot be applied to the semilinear stochastic heat equations. 
%{\b These conditions exclude many fucntions.} 
%
%Following this paper,
%Jentzen, Kloeden \& Winkel \cite{Jentzen-2011}
%proposed a simplified version of this scheme by 
%taking $s=t_{n-1}$ in the second term on the right hand of \eqref{fully2}, and higher convergence orders of {\b I can't find out what the orders provided in this paper are.}
%were given.
%
For the stochastic heat equations, Wang \& Qi \cite{Wang-2015} proved that the exponential Euler method in \cite{Jentzen-Kloeden-2009} has an $L^2$-norm error bound of %$\mathcal{O}(M^{-\frac14-}+\tau^{\frac12-})$
 $O(M^{-1}+\tau)$ and $O(M^{-\frac12+\varepsilon}+\tau^{\frac12-\varepsilon})$ for trace-class noises and one-dimensional space-time white noises, respectively.
 % (without additional conditions like $\|A^{\frac{\beta-1}{2}}Q^{\frac12}\|_{\mathcal{L}_2}<\infty$ for some $\beta>1$).
%
Recently, Wang \cite{Wang-2020} proposed a nonlinearity-tamed exponential integrator for the stochastic Allen--Cahn equation with a locally Lipschitz nonlinear drift function and proved an $L^2$-norm error bound of $O(M^{-\frac12+\varepsilon}+\tau^{\frac12-\varepsilon})$. 
%
%and $\mathcal{O}(M^{-\frac{1}{d}}+\tau)$ for trace-class noise
%much less restricting only general conditions.

%Combined either of the two methods, many fully discrete schemes for problem \eqref{SPDE1} have been investigated in recent years. 
%%
%Using the spectral method in space,
%%
%Wang \cite{Wang-2020}  proposed the nonlinearity-tamed accelerated exponential integrator scheme for the temporal discretization
%to numerically solve the stochastic Allen-Cahn equation disturbed by the space-time white noise, and obtained an order of $\frac12-$ in both space and time;
%%
%Using the exponential Euler method in time,
%%
 
%
In addition to the exponential Euler method, many other numerical methods for problem \eqref{SPDE1} have also been studied in the literature, including the semi-implicit Euler method in time and the finite difference/element method in space; see \cite{Brehier-Cui-Hong-2018, Gyongy-1998, Gyongy-1999, Mukam-Tambue-2020,Mukam-Tambue-2020-2, Wang-2017, Yan-2005}. 
The temporal convergence orders proved in these articles are not greater than $\frac12-\varepsilon$ for additive one-dimensional space-time white noises and not greater than $\frac14$ for multiplicative one-dimensional space-time white noises. 
The sharp order $\frac14$ for multiplicative space-time white noises was proved for sufficiently smooth initial value $u^0\in C^3(\overline\O)$ in \cite{Gyongy-1998, Gyongy-1999}. Recently, Anton, Cohen \& Quer-Sardanyons proved the sharp convergence order $\frac14$ for multiplicative space-time white noises noises for initial data only in $H^1(\O)$; see \cite{Anton-Cohen-2018}. 

For all the methods mentioned above, the suboptimal-order error bound $O(M^{-\frac12+\varepsilon}+\tau^{\frac12-\varepsilon})$ for additive one-dimensional space-time white noises was proved for initial data at least in $H^1(\O)$, and the sharper error bound $O(M^{-\frac12}+\tau^{\frac12})$ has not been proved yet. 
The main reason is that the following conditions were often used to characterize the regularity of the noises: 
\begin{align}\label{condition-past}
\Big\| \int_0^te^{-(t-s)A}\d W(s) \Big\|_{L^2(\Omega;L^2(\mathcal{O}))} \le C t^{\frac{\alpha}{2}}
\quad\mbox{and}\quad
\Big\| \int_0^te^{-(t-s)A}\d W(s) \Big\|_{L^2(\Omega;\dot H^\alpha(\mathcal{O}))}\le C ,
\end{align}
which can be satisfied by the space-time white noise with $\alpha=\frac12-\varepsilon$, but cannot be satisfied with the sharp order $\alpha=\frac12$. 

%We referred to \cite{Mukam-Tambue-2020, Feng-Li-Zhang-2017, Feng-Li-Zhang-2020, Cui-Hong-2019,
%Majee-Prohl-2018} for multiplicative noise for parabolic SPDEs could be
% 
%Extensive study have also been carried out for related equations to \eqref{SPDE1}. For example,

%There have also been plenty of results discussing
%numerical approximations to related equations, including 

Numerical approximations of related equations to \eqref{SPDE1} were also extensively studied. For example,
different kinds of
stochastic differential equations were considered in 
%(e.g.,
\cite{Higham-2002, Hutzenthaler-2011, Hutzenthaler-2012, A-Lang-2017, Cao-2018, Hutzenthaler-2020},
%)
and the semilinear stochastic wave equations were discussed in 
 \cite{Cao-2007, Wang-Gan-Tang-2014, Anton-Cohen-Larsson-Wang-2016, Qi-Wang-2019, Kovas-Larsson-Saedpanah-2010, Banjai-Lord-Molla-2021, A-Lang-2021}, where convergence rates are higher than that for the semilinear stochastic heat equation
 due to the better regularity of the solution to the stochastic wave equations. 

In this article, we show that by modifying the exponential Euler method at the starting time level and use proper variable stepsizes locally refined towards $t=0$, the temporal and spatial convergence orders of the numerical solution can be improved to $\frac12$ for additive one-dimensional space-time white noises, while the regularity of the initial data can be relaxed to 
$H^{\beta}(\O)$ with $\beta>-1$.
%{\color{red}$u^0\in L^p(\Omega;H^{\beta}(\O))$ with any $\beta>-1$ for all $p\in[2,4]$.} 
This wider class of initial data includes discontinuous functions and measures in one dimension, such as the Dirac delta measure. In particular, the following error bound is proved: 
%
%Our first objective is to see whether
%the convergence orders for the case of space-time white noise could be improved to exactly $\frac12$ for both spatial and temporal approximations.
%
%With our main result in Theorem \ref{main-theorem}, i.e., simply
\begin{align}\label{Result-introduction}
\|U_M^n-u(t_n)\|_{L^2(\Omega;L^2(\mathcal{O}))}
\le
C(\tau^{\alpha}+M^{-\alpha}t_n^{-\frac{\alpha-\beta}{2}})
\quad\mbox{for}\,\,\, t_n\in(0,T] ,
\end{align}
%for $u^0\in \dot{H}^{\beta}(\O)$ with any $\beta\in(-1,\alpha]$, 
for $u^0\in L^4(\Omega;H^{\beta}(\O))$ with some constant $\beta>-1$, 
where $\alpha$ characterizes the regularity of the noise, as described in Section \ref{section:assumptions}. 
%For example, $\alpha=1$ for trace-class noises and $\alpha=\frac12$ for one-dimensional space-time white noises. 
The result in \eqref{Result-introduction}, which includes the sharp order $\alpha=\frac12$ for one-dimensional space-time white noises, is obtained by using a class of locally refined variable stepsizes to resolve the singularity of the solution at $t=0$, and by utilizing the real interpolation properties of a class of stochastic  Besov spaces $B^{q}L^p(\Om;\dot{H}^{s}(\O))$, $s\in\R$, $1\le p,q\le \infty$, defined in this article, replacing the condition \eqref{condition-past} by  
\begin{align}\label{condition-Besov}
\Big\| \int_0^te^{-(t-s)A}\d W(s) \Big\|_{L^2(\Omega;L^2(\mathcal{O}))} \le C t^{\frac{\alpha}{2}}
\quad\mbox{and}\quad
\Big\| \int_0^te^{-(t-s)A}\d W(s) \Big\|_{B^\infty L^2(\Omega;\dot H^\alpha(\mathcal{O}))}\le C,
\end{align}
which incorporates the space-time white noise with $\alpha=\frac12$.

The rest of this article is organized as follows. 
In Section \ref{assumption}, we first introduce the basic notations to be used in this article, as well as the definition and properties of the stochastic Besov spaces.
Then we present the numerical methods and the main theoretical results. 
%Then the main results of this paper including the assumptions for the drift term and noise, 
%our used numerical schemes and the main theorem (Theorem \ref{main-theorem}), are provided in the remaining part of  Section \ref{assumption}. 
%
The proof of the main theorem is presented in Section \ref{proof_of_main}.
%, which contains 
%
Numerical results for several different initial data and noises are presented in Section \ref{numerical} to support the theoretical analysis.
Conclusions are presented in Section \ref{conclusions}.

\section{Main results}\label{assumption}
%\begin{remark}
%{\upshape
%For example, if $A$ and $Q$ have the same eigenfunctions, with $\lambda_i$ and $\phi_i$, $i=1,2,\dots,$ being the eigenvalues and eigenfunctions of $A$, 
%then 
%\begin{align}
%\int_{t_{n-1}}^{t_n} e^{-(t_n-s)A} \d W(s) 
%&= \sum_j \sqrt{\mu_j} \int_{t_{n-1}}^{t_n} e^{-(t_n-s)A}  \phi_j \d W_j(s) \notag \\ 
%&= \sum_j \sqrt{\mu_j} \int_{t_{n-1}}^{t_n} e^{-(t_n-s)\lambda_j}  \phi_j\d W_j(s) \notag \\ 
%&= \sum_j \sqrt{\mu_j} X_j \phi_j  
%\end{align}
%with
%$$
%X_j = \int_{t_{n-1}}^{t_n} e^{-(t_n-s)\lambda_j} \d W_j(s) \quad j=1,2,\dots
%$$
%are independent Gaussian random variables such that 
%$$
%\mathbb{E}(X_j) = 0
%\quad\mbox{and}\quad 
%\mathbb{E}(X_j^2) = \int_{t_{n-1}}^{t_n} e^{-2(t_n-s)\lambda_j} \d s 
%= \lambda_j^{-1}(1-e^{-2\tau_n\lambda_j}) . 
%$$}
%\end{remark}

\subsection{Basic notations}\label{Basic_notations}
%We first provide definitions of conventional Sobolev spaces as follows:
%$$
%\|v\|_{L^p(\O)}:=\begin{cases}
%\big(\int_{\O}|v|^p\d x\big)^{\frac{1}{p}} \qquad\  1\le p<\infty,\\
%\mbox{ess\, sup}_{x\in \O} |v(x)| \quad p=\infty,
%\end{cases} 
%$$
%and
%$$
%\|v\|_{H^{k}(\O)}:=%\begin{cases}
%\Big(\displaystyle\sum_{|\alpha|\le k}\|D^{\alpha}v\|_{L^2(\O)}^2\Big)^{\frac{1}{2}},  %&  1\le p<\infty,\\
%%\displaystyle\max_{|\alpha|\le k}  \|D^{\alpha}v\|_{L^\infty} & p=\infty,
%%\end{cases} 
%$$
%where $D^{\alpha}=\frac{\pa^{|\alpha|}}{\pa x_1^{\alpha_1}\dots \pa x_d^{\alpha_d}}$ with $\alpha=(\alpha_1,\cdots,\alpha_d)$ and $|\alpha|=\alpha_1+\cdots+\alpha_d$ $(\alpha_i\ge 0, i=1,\cdots,d)$.
%For $1\le p\le \infty$ and $s\in\R$, we denote by $L^p(\O)$ and $\dot H^s(\O)$ the conventional Lebesgue and Sobolv spaces on a bounded domain $\O\subset\R^d$.  
%Similarly, we have the space $\dot{H}^s(\O)$ defined by the operator $A$ as
%with respect to the operator $A$ for all $s\in \mathbb{R}$.
For {$s\in[0,2]$}, we denote by $A^{\frac{s}{2}}:D(A^{\frac{s}{2}})\rightarrow L^2(\O)$ the linear operator with domain 
$$
D(A^{\frac{s}{2}})= \Big\{v= \sum_{k=1}^\infty v_k\phi_k:  \|A^{\frac{s}{2}}v\|_{L^2(\O)}^2=\sum_{k=1}^\infty \lambda_k^{s}  v_k^2 <\infty \Big\} ,
$$
where
$$
A^{\frac{s}{2}}v := \sum_{k=1}^{\infty} \lambda_k^{\frac{s}{2}} v_k\phi_k .
%\quad 
%\mbox{and} 
%\quad
%\|v\|_{\dot{H}^s(\O)}=\|A^{\frac{s}{2}}v\|_{L^2(\O)}.
$$
It is known that $D(A^{\frac{s}{2}})$ coincides with the following real interpolation space with equivalent norms:
$$
D(A^{\frac{s}{2}}) \cong \dot{H}^s(\O):= (L^2(\O),H^1_0(\O)\cap H^2(\O))_{\frac{s}{2},2} \quad\mbox{for}\,\,\, s\in(0,2) . 
$$
Therefore, we simply define the norm of $\dot{H}^s(\O)$ to be the same as $D(A^{\frac{s}{2}})$, i.e.,
$$
\|v\|_{\dot{H}^{s}(\O)} =\Big(\sum_{k=1}^{\infty}\lambda_k^{s}v_k^2\Big)^{\frac12} .
$$ 
The operator $A^{\frac{s}{2}}: \dot H^s(\O)\rightarrow L^2(\O)$ is obviously invertible, and its inverse is given by  
\be\label{A_minus_s}
A^{-\frac{s}{2}}v=\sum_{k=1}^{\infty}\lambda_k^{-\frac{s}{2}}v_k\phi_k . 
\en

The dual space of $\dot{H}^s(\O)$ is denoted by $\dot{H}^{-s}(\O)$. 
In particular, $v\in \dot{H}^{-s}(\O)$ if and only if $v=\sum_{k=1}^{\infty}v_k \phi_k$ with 
$$
\|v\|_{\dot{H}^{-s}(\O)}=\Big(\sum_{k=1}^{\infty}\lambda_k^{-s}v_k^2\Big)^{\frac12}< \infty.
$$ 

{%\color{blue}
We denote by $\langle\cdot,\cdot\rangle_{\dot{H}^s,\dot{H}^{-s}}$ the pairing between $\dot{H}^{s}(\O)$ and $\dot{H}^{-s}(\O)$. Namely, for $g=\sum_{k=1}^\infty g_k\phi_k\in \dot{H}^s(\O)$ and $h=\sum_{k=1}^\infty h_k\phi_k\in \dot{H}^{-s}(\O)$, we have 
$$
\langle g, h\rangle_{\dot{H}^s(\O),\dot{H}^{-s}(\O)}
:=
\sum_{k=1}^{\infty}g_kh_k ,
$$
which is well-defined as the series $\sum_{k=1}^{\infty}g_kh_k $ is absolutely convergent, i.e.,  
$$
\sum_{k=1}^{\infty} |g_kh_k|
=\sum_{k=1}^{\infty}\lambda_k^{\frac{s}{2}}|g_k|\lambda_k^{-\frac{s}{2}}|h_k|
%=\sum_{k=1}^{\infty}\lambda_k^{\frac{s}{2}}u_k\lambda_k^{-\frac{s}{2}}v_k(\phi_k,\phi_k)
\le \Big(\sum_{k=1}^{\infty}\lambda_k^{s}g_k^2\Big)^{\frac12}
\Big(\sum_{k=1}^{\infty}\lambda_k^{-s}h_k^2\Big)^{\frac12}
= \|g\|_{\dot{H}^{s}(\O)}\|h\|_{\dot{H}^{-s}(\O)}.
$$
In the case $g\in \dot{H}^s(\O)\subset L^2(\O)$ and $h\in L^2(\O)\subset  \dot{H}^{-s}(\O)$, we have 
$$\langle g, h\rangle_{\dot{H}^s(\O),\dot{H}^{-s}(\O)} = (g,h) , $$
where $(\cdot,\cdot)$ denotes the inner product of $L^2(\Omega)$. Therefore, $\langle \cdot, \cdot\rangle_{\dot{H}^s(\O),\dot{H}^{-s}(\O)}$ is an extension of the $L^2(\O)$ inner product $(\cdot,\cdot)$. For this reason and the simplicity of notation, we simply use $(\cdot,\cdot)$ to denote the pairing between $\dot{H}^s(\O)$ and $\dot{H}^{-s}(\O)$ in this article. 
}

The operator $A^{-\frac{s}{2}}: \dot H^{-s}(\O)\rightarrow L^2(\O)$ is well defined {by \eqref{A_minus_s}} and coincident with the adjoint operator of $A^{-\frac{s}{2}}: L^2(\O)\rightarrow \dot H^s(\O)$, i.e.,
\begin{align*}
\hspace{50pt}
(u,A^{-\frac{s}{2}}w) 
=&
\Big(\sum_{i=1}^{\infty}u_i\phi_i,\sum_{j=1}^{\infty}\lambda_i^{-\frac{s}{2}}w_j\phi_j\Big)
\\
=&
\sum_{i=1}^{\infty}\lambda_i^{-\frac{s}{2}}u_iw_i 
=
(A^{-\frac{s}{2}}u,w)
\qquad\forall\, u\in L^2(\O) \,\,\mbox{and}\,\, w\in \dot{H}^{-s}(\O),
\end{align*}
{where the right-hand side denotes the pairing between $\dot{H}^s(\O)$ and $\dot{H}^{-s}(\O)$.} 

%%Since $A^{-\frac{s}{2}}: L^2(\O)\rightarrow \dot H^s(\O)$ is self-adjoint, its extension $A^{-\frac{s}{2}}: \dot H^{-s}(\O)\rightarrow L^2(\O) $ is well defined and given by }}
%%%For $v\in \dot{H}^{-s}(\O)$ we could analogously define
%$$
%A^{-\frac{s}{2}}v=\sum_{k=1}^{\infty}\lambda_k^{-\frac{s}{2}}v_k\phi_k . 
%%\quad 
%%\mbox{and therefore}
%%\quad 
%%\|v\|_{\dot{H}^{-s}(\O)}=\|A^{-\frac{s}{2}}v\|_{L^2(\O)}.
%$$

For a random variable $v$ that takes values in a Banach space $X$ and is measurable from $(\Omega,\mathcal{F})$ to $(X,\mathcal{B}(X))$,
%i.e., $\mathcal{F}/\mathcal{B}(X)$-measurable
 we define the following norm: 
\begin{align*}
\|v\|_{L^p(\Om;X)}
:=\Big(\int_{\Om}\|v\|_X^{p} \, P(\d\omega)\Big)^{\frac{1}{p}}
\qquad\forall\, 1\le p \le \infty. 
\end{align*} 
If we denote by $\|\cdot\|_{L^2(\O)\to L^2(\O)}$ the operator norm on $L^2(\O)$, then the following estimate holds (see \cite[Appendix B.2]{Kruse-2014}): 
\begin{align*}
&\|A^{s}e^{-tA}\|_{L^2(\O)\to L^2(\O)}
\le Ct^{-s} 
&&\forall\,  t>0 ,\,\,\, \forall\, s\ge 0 ,\\% \label{A-e} 
%%%%%%%%%%%%%%%%%%%%
&\|A^{-s}(e^{-tA}-I)\|_{L^2(\O)\to L^2(\O)}
\le Ct^{s} 
&&\forall\,  t\ge 0 ,\,\,\, \forall\, s\in[0,1]  %\label{A-I-e}
\end{align*}

Let $\mathcal{L}_2$ be the space of Hilbert Schmidt operators $\Phi$ from $L^2(\O)$ to $L^2(\O)$, with the following norm: 
$$
\|\Phi\|_{\mathcal{L}_2}:=\Big(\sum_{k=1}^\infty \|\Phi\phi_k\|_{L^2(\O)}^2\Big)^{\frac12},
$$
Let $\mathcal{L}_2^0$ be the space of Hilbert Schmidt operators $\Psi$ from $Q^{\frac12}L^2(\O)$ to $L^2(\O)$, with the following norm: 
%$$
%{\color{blue}{
%\|\Psi\|_{\mathcal{L}_2^0}
%=\|\Psi\circ Q^{\frac12}\|_{\mathcal{L}_2}
%=\Big(\sum_k \mu_k \|\Psi\phi_k\|_{L^2(\O)}^2\Big)^{\frac12} .}}
%$$
$$
\|\Psi\|_{\mathcal{L}_2^0}:=\Big(\sum_k \mu_k \|\Psi\phi_k\|_{L^2(\O)}^2\Big)^{\frac12},
$$
{where we have adopted the notation $\mathcal{L}_2^0$  and $Q^{\frac12}L^2(\O)$ in \cite[pages 54 and 96]{Da-Prato1992}.} 
%where $\Psi$ is a  Hilbert Schmidt operator from $Q^{\frac12}L^2(\O)$ to $L^2(\O)$.
Let $\mathcal{N}_W^2(0,T;L^2(\O))$ be the space of predictable processes $\Phi:[0,T]\times \Om\to \mathcal{L}_2^0$ satisfying
$$
%\|\Phi\|_{t}
%=
\int_0^T\|\Phi(s)\|_{L^2(\Om;\mathcal{L}_2^0)}^2\d s<\infty . 
$$
{ For $0\le t\le T$, it} is known that the following It\^o's isometry holds (cf. \cite{Catherine-2014,Rockner-2007}):
%\cite[Section 2.3]{Rockner-2007}, \cite[Theorem 10.16]{Catherine-2014}): 
\begin{align}\label{ito-isometry}
\Big\|\int_0^t\Phi(s)\d W(s)\Big\|_{L^2(\Om;L^2(\O))}^2
=
\int_0^t\|\Phi(s)\|_{L^2(\Om;\mathcal{L}_2^0)}^2\d s
\quad\forall\, \Phi\in \mathcal{N}_W^2(0,T;L^2(\O)),
\end{align}
{ and the following Burkholder–Davis–Gundy-type inequality holds (cf. \cite{Da-Prato1992,Kruse-2014})
\begin{align}\label{BDG}
\Big\|\int_0^t\Phi(s)\d W(s)\Big\|_{L^p(\Om;L^2(\O))}
\le C_p 
\Big(\mathbb{E}\int_0^t \|\Phi(s)\|_{\mathcal{L}_2^0}^2\d s\Big)^{\frac12}\quad\forall\, \Phi\in \mathcal{N}_W^2(0,T;L^2(\O)),
%\Big(\int_0^t \|\Phi(s)\|_{L^p(\Om;\mathcal{L}_2^0)}^2\d s\Big)^{\frac12}
%\Big[\Big(\int_0^t\|\Phi(s)\|_{\mathcal{L}_2^0}^2\d s\Big)^{\frac{p}{2}}\Big]
\end{align}
where $C_p$ is a constant dependent of $p$. }
% In detail:
% \cite[Page 174, Hypothesis 6.4]{Da-Prato1992} and \cite[Page 18, Proposition 2.12]{Kruse-2014} 
%for all $\Phi\in \mathcal{N}_W^2(0,t;L^2(\O))$ and $0\le t\le T$. 

Throughout this article, we denote by $C$ a generic positive constant which may be different at different occurrences but always independent of $\tau$ (time stepsize), $n$ (time level) and $M$ (degrees of freedom in each spatial direction). 
We denote by ``$A\sim B$' the statement ``$C^{-1}B\le A\le CB$ for some constant $C$''.  

\subsection{Stochastic Besov Spaces}\label{Stochastic_Besov_Spaces}
Let $1\le p,q\le \infty$ and $s \in[-2,2]$. Since any function $v\in L^p(\Omega; \dot{H}^{-2}(\O))$ can be decomposed into 
$$v=\sum_{k=1}^{\infty}v_k\phi_k \quad\mbox{with}\quad v_k= (v,\phi_k) \in L^p(\Omega) , $$ 
we can define a projection operator $\Pi_j: L^p(\Omega;\dot{H}^{-2}(\O)) \rightarrow L^p(\Omega;\dot H^2(\O)) $ by 
$$\Pi_jv=\sum_{k=2^{j-1}}^{2^j-1}v_k\phi_k . $$ 
% where 
%$$I_j=\{k\in\mathbb{N}^+: 2^{j-1}\le k < 2^{j}\}.$$ 
%We should point out that $\Pi_jv=0$ if $j\le 0$.
Then the stochastic Besov space $B^qL^p(\Omega;\dot{H}^{s}(\mathcal{O}))$ is defined as the space of functions $v\in L^p(\Omega; \dot{H}^{-2}(\O))$ such that  
%is given to be 
%the space such that $\forall v\in B^qL^p(\Omega;\dot{H}^{s}(\mathcal{O}))$ it satisfies
$$\|v\|_{B^qL^p(\Omega;\dot{H}^{s}(\mathcal{O}))}<\infty , $$
with
\begin{align*}
\|v\|_{B^qL^p(\Omega;\dot{H}^{s}(\mathcal{O}))}
:=
\left\{
\begin{aligned}
&\Big(\sum\limits_{j=1}^{\infty} \|\Pi_jv\|_{L^p(\Omega;\dot{H}^{s}(\mathcal{O})}^q\Big)^{\frac1q}
&& \mbox{if}\; q\in [1,\infty),\\
&\max\limits_{j\in \mathbb{N}^+}\|\Pi_jv\|_{L^p(\Omega;\dot{H}^{s}(\mathcal{O}))}&& \mbox{if}\; q=\infty.
\end{aligned}
\right.
\end{align*}

\subsection{Assumptions on the nonlinearity and noise}
\label{section:assumptions}

We consider the semilinear stochastic heat equation \eqref{SPDE1} with additive noises under the following assumptions. 
\begin{assumption}\label{thm-assumption}
%\mbox{}
\begin{enumerate}
\item[(1)]
%$f:L^2(\O)\rightarrow L^2(\O)$ is {\color{red} globally?} Liptschitz continuous, i.e.,
%\begin{align*}
%\|f'(u)v\|_{L^2(\O)} 
%&\le
% C  \|v\|_{L^2(\O)}, \quad\forall\, u,v \in L^2(\O).   \\
% \|f(u)\|_{L^2(\O)} 
% &\le 
% C\|u\|_{L^2(\O)},\quad\forall\, u,v \in L^2(\O).  
%\end{align*}
The function $f:\mathbb{R}\to\mathbb{R}$ is globally Lipschitz continuous,
i.e.,% {\color{red}Is locally continuous enough?}
\begin{align*}
|f(x)-f(y)|\le C|x-y|, \quad\forall\, x,y\in\mathbb{R}.
\end{align*} 
\item[(2)]
There exists $\eta\in(0,2)$ such that 
$$ 
\| (-A)^{-\frac{\eta}{2}}[{ f''(u)vw}] \|_{L^2(\O)} \le C \|v\|_{L^2(\O)} \|w\|_{L^2(\O)} \quad\forall\, u,v,w\in L^2(\O). 
$$
\end{enumerate}
%For a sequence $\sqrt{\mu}=(\sqrt{\mu_k})_{k=1}^{\infty}$, we define 
%%\begin{align*}
%%\|\sqrt{\mu}\|_{B^qH^m}=\Big(\sum_{j=1}^{\infty}(2^{jm}\sum_{k=2^{j-1}}^{2^j-1}\mu_k)^{\frac{q}{2}}\Big)^{\frac{1}{q}}.
%%\end{align*}
%\begin{align*}
%\|\sqrt{\mu}\|_{B^{\infty}\Lambda}=\max_{j}(\sum_{k=2^{j-1}}^{2^j-1}\lambda_k^{\alpha-1}\mu_k)^{\frac{1}{2}},
%\end{align*}
%which is evidently a norm.
%Also we require the assumptions for the noise term:
\begin{enumerate}
\item[(3)]
There exists an $\alpha\in (0,1]$ such that for $t\in [0,T]$
%{\color{red} Why can't $\alpha$ be 0? Not necessary? Since it's a denominator in the restriction condition of $\gamma$?}
\begin{align}
&\Big\| \int_0^te^{-(t-s)A}\d W(s) \Big\|_{L^2(\Omega;L^2(\mathcal{O}))} \le C t^{\frac{\alpha}{2}}, \label{t-alpha}\\
&\Big\| \int_0^te^{-(t-s)A}\d W(s) \Big\|_{B^\infty L^2(\Omega;\dot H^\alpha(\mathcal{O}))} 
%%\|(\sqrt{\mu_k})\|_{B^{\infty}\Lambda_{\alpha}}
%\le  C \max_{j\in\mathbb{N}^+} \Big(\sum_{k=2^{j-1}}^{2^j-1} \mu_k \lambda_k^{\alpha-1}(1 - e^{-2t\lambda_k})\Big)^{\frac{1}{2}}
{\le C}.\label{mu-C}
\end{align}
%{\color{red}Shall we change the number ``2'' in $L^2(\Om;)$ to p?}
{ In the case $\alpha=1$, we additionally assume that the noise is trace class.}
\item[(4)] 
The initial value is $\mathcal{F}_0$ measurable and satisfies that 
{$u^0\in L^4(\Omega;\dot{H}^{\beta}(\O))$  for some constant $\beta\in(-1,\alpha]$}. 
%$u^0\in \dot{H}^{\beta}(\O)$ with $\beta\in(-1,\alpha]$. 
%\ben
%-\alpha\le\beta<2-\alpha\quad %0\le\beta_0 <2-\alpha,\quad
%\mbox{and}\quad
%\max\{\frac{\beta}{2},\frac12,\alpha+\beta-1\}<\gamma<1,
%\enn
\end{enumerate}
\end{assumption}

\begin{remark}\label{noise-remark}
\upshape 
%Assumption (2)
{
Assumption \ref{thm-assumption} (2) holds for $d\in\{1,2,3\}$ if $f:\R\rightarrow\R$ is a function with bounded derivatives up to second order; see \cite[Example 3.2]{Wang-2015-2}. 
In one dimension, %the
a large number of measure-valued functions %(including the Dirac delta function) 
satisfy %Assumption (4). 
Assumption \ref{thm-assumption} (4). Actually, each measure $\mu$ 
(including the Dirac delta function,
i.e., the Dirac measure) 
corresponds to a linearly bounded  functional $\Lambda(f):=\int_{\O}f\d \mu$ on the continuous function space $C(\O)$; see
\cite[page 61]{Malliavin-1995} and
\cite[page 32%, Theorem 7.3
]{Stein-functional}. And it is well-known that $\dot{H}^{\frac12+\varepsilon}(\O)\hookrightarrow C(\O)$; see \cite[page 86, Theorem 3.26
]{Mclean-2000}. 
Therefore $\mu$ can regarded as a linearly bounded functional on $\dot{H}^{\frac12+\varepsilon}(\O)$, i.e.,  
 $\mu\in \dot{H}^{-\frac12-\varepsilon}(\O)$.}
%Assumption (3)
Assumption \ref{thm-assumption} (3) naturally holds for trace-class noises with $\alpha=1$ and one-dimensional space-time white noises with $\alpha=1/2$, as shown below.

%
% $\dot{H}^{\frac12+\epsilon}(\O)$ which embeds into $C(\O)$
%hence each measure-valued function belongs to $L^p(\dot{H}^{-(\frac12+\epsilon)}(\O))$
%
%
%For example,  the Dirac delta function belongs to $\dot{H}^{-\frac12-\eps}(\O)$ with an arbitrary small positive constant $\varepsilon>0$.

% 
%

\begin{enumerate}
\item 
If the operator $Q$ is of trace class, i.e., ${\rm tr}(Q)=\sum\limits_{k=1}^\infty\mu_k < \infty$, then %Assumption (3) 
Assumption \ref{thm-assumption} (3) 
holds with $\alpha=1$. 
In fact, 
%for any $p\ge 2$ 
according to \eqref{ito-isometry} 
%and the result in \cite[Page 18, Proposition 2.12]{Kruse-2014} 
{ the following relation holds}
% (cf. \cite[Page 18, Proposition 2.12]{Kruse-2014}) for $p\ge 2$: 
\begin{align}\label{Noise-Expansion}
\Big\| \int_0^te^{-(t-s)A}\d W(s) \Big\|^2_{ L^2(\Omega;L^2(\mathcal{O}))}
&{ =}
\int_0^t \|e^{-(t-s)A}\|_{\mathcal{L}_2^0}^2\d s\notag  \\
&{ =}%\sim
\sum_{k=1}^\infty \int_0^t e^{-2(t-s)\lambda_k} \mu_k \d s \notag \\
&{ =}%\sim
\sum_{k=1}^\infty \mu_k \frac{1 - e^{-2t\lambda_k}}{2\lambda_k} 
\le Ct \sum_{k=1}^\infty \mu_k . 
\end{align}
%This proves \eqref{t-alpha} with $\alpha=1$. 

The equivalence relation in {\eqref{mu-C}} can be shown similarly as \eqref{Noise-Expansion}, i.e., %for $p\ge 2$
\begin{align}\label{classical-alpha=1}
\Big\|\int_0^{t}e^{-(t-s)A}\d W(s)\Big\|_{ L^2(\Omega;\dot H^1(\mathcal{O}))}
=&\ 
\Big\|\int_0^{t}A^{\frac{1}{2}}e^{-(t-s)A}\d W(s)\Big\|^2_{L^2(\Omega;L^2(\mathcal{O}))} \notag\\
=%\sim
&\
\int_0^t \|A^{\frac{1}{2}}e^{-(t-s)A}\|_{\mathcal{L}_2^0}^2\d s
\notag\\
= %\sim
&\
\sum\limits_{k=1}^\infty \frac{\mu_k}{2} (1 - e^{-2t\lambda_k})\notag\\
\le%\sim
&\
\sum\limits_{k=1}^\infty \mu_k , %\quad\mbox{in the case $\alpha=1$},
\end{align}
{ which implies that 
\begin{align*}
\Big\|\Pi_j\int_0^{t}e^{-(t-s)A}\d W(s)\Big\|_{ L^2(\Omega;\dot H^1(\mathcal{O}))}
\le
\Big\|\int_0^{t}e^{-(t-s)A}\d W(s)\Big\|_{ L^2(\Omega;\dot H^1(\mathcal{O}))}
\le %\sim
&\
\sum\limits_{k=1}^\infty \mu_k . 
\end{align*}
This proves \eqref{t-alpha}--\eqref{mu-C} in the case $\alpha=1$.}

\item
If $d=1$ and $\mu_k= 1$ (for a space-time white noise), then 
Assumption \ref{thm-assumption} (3) 
%Assumption (3)
 holds with $\alpha=\frac12$. In fact, Weyl's law (see Evans \cite[Page 358]{Evans}) says that the eigenvalues of the Laplacian operator have the following asymptotic behaviour: 
\begin{align}\label{weyl}
\lim_{j\to \infty}\frac{\lambda_j^{\frac{d}{2}}}{j}=\frac{(2\pi)^d}{|\mathcal{O}|\alpha(d)},
\end{align} 
where $|\O|$ denotes the volume of $\O$ and $\alpha(d)$ denotes the volume of the unit ball in $\mathbb{R}^d$. 
Therefore, $\lambda_k=O(k^{-2})$ in one dimension, {and} %for $p\ge 2$}
\begin{align}
\Big\| \int_0^te^{-(t-s)A}\d W(s) \Big\|^2_{L^2(\Omega;L^2(\mathcal{O}))}
&\sim C (1-e^{-2t}) + C \sum_{k=2}^{\infty}\frac{1-e^{-2tk^2}}{k^{2}}\notag\\
&\le Ct^{\frac{1}{2}} + 
C\int_1^{\infty}\frac{1-e^{-2ts^2}}{s^{2}}ds\notag\\
&= Ct^{\frac{1}{2}} +  
C\int_1^{\infty}\frac{1-e^{-2tz}}{z^{\frac{3}{2}}}dz\notag\\
%&= Ct^{\frac{1}{2}} 
%-C\int_1^{\infty}(1-e^{-2tz})dz^{-\frac{1}{2}}\notag\\
%&= Ct^{\frac{1}{2}} +  
%C(1-e^{-2t})+Ct^{\frac{1}{2}}\int_1^{\infty}(zt)^{-\frac{1}{2}}e^{-2tz}d(zt)\notag\\
&\le
Ct^{\frac{1}{2}}+Ct^{\frac{1}{2}}\Gamma\Big(\frac{1}{2}\Big) 
\le
Ct^{\frac{1}{2}} . \label{little-delight}
\end{align}
%%which implies $\alpha = \frac12$ in the case of $\mu_k \equiv 1$ in one
%
%
%
%In the case $d=1$ and $\mu_j=1$, \eqref{Noise-Expansion} and imply that 
%\ben
%\begin{aligned}
%\Big\| \int_0^te^{-(t-s)A}\d W(s) \Big\|_{L^p(\Omega;L^2(\mathcal{O}))}^2
%&\sim (1-e^{-2t}) + \sum_{k=2}^{\infty}\frac{1-e^{-2tk^2}}{k^2} \\
%&\le Ct + 
%\int_1^{\infty}\frac{1-e^{-2ts^2}}{s^2}ds\\
%&= Ct +  
%\frac12\int_1^{\infty}\frac{1-e^{-2tz}}{z^{\frac32}}dz\\
%&= Ct 
%-\int_1^{\infty}(1-e^{-2tz})dz^{-\frac12}\\
%&= Ct +  
%(1-e^{-2t})+2t^{\frac12}\int_1^{\infty}(zt)^{-\frac12}e^{-2tz}d(zt)\\
%&\le
%Ct+4t^{\frac12}\Gamma(\frac12)\\
%&\le
%Ct^{\frac12} . 
%\end{aligned}
%\enn
This proves \eqref{t-alpha} with $\alpha=\frac12$ for a space-time white noise. 
%{\color{red}(Why do we have this paragraph?)}
%Especially, when $\delta=1$ we could only obtain 
%\be\label{little-sorrow}
%\begin{aligned}
%\Big\| \int_0^te^{-(t-s)A}\d W(s) \Big\|_{L^p(\Omega;L^2(\mathcal{O}))}^2
%&\sim (1-e^{-2t}) +  \sum_{k=2}^{\infty}\frac{1-e^{-2tk^2}}{k^{3}}\\
%&\le Ct + 
%\int_1^{\infty}\frac{1-e^{-2ts^2}}{s^{3}}ds\\
%&= Ct+  
%C\int_1^{\infty}\frac{t^{1-\eps}s^{2-2\eps}}{s^{3}}ds\\
%&\le
%Ct^{1-\eps},
%\end{aligned}
%\en
%where $\eps>0$ is any a sufficiently small constant. 

For $\alpha=\frac12$ we also see that 
\begin{align*}
\max_{j\in \mathbb{N}^+} \Big(\sum_{k=2^{j-1}}^{2^j-1} \mu_k \lambda_k^{\alpha-1} (1 - e^{-2t\lambda_k}) \Big)^{\frac{1}{2}}
\le 
\max_{j\in \mathbb{N}^+}\Big(C\sum_{k=2^{j-1}}^{2^j-1}{2^{2(j-1)(\frac12-1)}}\Big)^{\frac12}
\le
 C, 
\end{align*}
where the last inequality holds because there are only $2^{j-1}$ terms in the summation. 
This proves \eqref{mu-C}. 
%when $d=1$. 
\end{enumerate}
\end{remark}
%\begin{remark}
%By means of the result in \cite[Page 18, Proposition 2.12]{Kruse-2014}, for any $p\ge 2$ we can directly 
%have 
%\be\label{Noise-Expansion-2}
%\Big\| \int_0^te^{-(t-s)A}\d W(s) \Big\|_{L^p(\Omega;L^2(\mathcal{O}))}^2 
%\le
%C\int_0^t \|e^{-(t-s)A}\|_{\mathcal{L}_2^0}^2\d s, 
%\en
%which shows that under Assumption (3)
% the inequalities \eqref{t-alpha} and \eqref{mu-C} still hold even if
% the norms on the left-hand sides are replaced by 
%$L^p(\Om;L^2(\O))$ and $B^{\infty}L^p(\Om;\dot{H}^{\alpha}(\O))$ respectively.
%\end{remark}
%\begin{remark}
%\upshape \textit{
%the assumption \eqref{t-alpha} could be actually replaced by 
%\begin{align}
%\Big\| \int_0^te^{-(t-s)A}\d W(s) \Big\|_{L^2(\Omega;L^2(\mathcal{O}))}^2 &\le C t^{\alpha}, \label{t-alpha-L2}
%\end{align}
%since 
%$$
%\Big\| \int_0^te^{-(t-s)A}\d W(s) \Big\|_{L^4(\Omega;L^2(\mathcal{O}))}^2\sim\Big\| \int_0^te^{-(t-s)A}\d W(s) \Big\|_{L^2(\Omega;L^2(\mathcal{O}))}^2.
%$$}
%\end{remark}

Under %Assumptions (1)--(4), 
Assumption \ref{thm-assumption} 
the existence, uniqueness and regularity of mild solutions to problem \eqref{SPDE1} are summarized below. The {proof} of these results is presented in Appendix. 
\begin{proposition}\label{Regularity} 
Under %Assumptions (1)--(4),
Assumption \ref{thm-assumption}, 
 problem \eqref{SPDE1} has a unique mild solution $u$ in the space 
$$
X=\Big\{v\in L^1\big(0,T; L^2(\Om;L^2(\O))\big): \sup\limits_{t\in (0,T]} (1+t^{\frac{\beta}{2}})^{-1} \|v(t)\|_{L^2(\Om;L^2(\O))}<\infty \Big \}. 
$$
%{\textcolor{blue}{The mild solution has the following qualitative regularity: 
%$$u\in C\big([0,T]; L^p(\Om;\dot{H}^{\beta}(\O))\big) \cap C\big([\varepsilon,T]; L^p(\Om;L^2(\O))\big)
%\cap C\big([\varepsilon,T]; B^\infty L^p(\Om;\dot H^\alpha(\O))\big) , $$
%which holds for arbitrary $\varepsilon\in(0,T)$. Moreover, the following quantitative estimates hold:
%\begin{align}\label{u-L2-estimate}
%\|u(t)\|_{L^p(\Omega;L^2(\mathcal{O}))}
%&\le
%C(1+t^{\frac{\beta}{2}}) \hspace{20pt} \mbox{for}\,\,\, t\in(0,T] , \\ 
%\label{B-Inf-estimate}
%\|u(t)\|_{B^{\infty}L^p(\Om;\dot{H}^{\alpha}(\O))}
%&\le
%Ct^{-\frac{\alpha-\beta}{2}} \hspace{30pt} \mbox{for}\,\,\, t\in(0,T] . 
%\end{align}}}
%{\textcolor{red}{ I have changed the $``p"$ in $\|\cdot\|_{L^p(\Omega;L^2(\O))}$
%to $``4"$ in this proposition and the proof for this proposition, so that the assumption of the initial data can just be $u^0\in L^4(\Omega;\dot{H}^{\beta})$ instead of 
%$u^0\in L^p(\Omega;\dot{H}^{\beta})$ with $p\ge 2$.}}
%
For {$2\le p\le 4$} %and $\bar{\beta}:=\min\{\beta,0\}$} 
the mild solution has the following qualitative regularity: 
$$u\in C\big([0,T]; L^p(\Om;\dot{H}^{\min\{\beta,0\}}(\O))\big) \cap C\big([\varepsilon,T]; L^p(\Om;L^2(\O))\big)$$
%{\b with $\bar{\beta}:=\min\{\beta,0\}$},
%\cap C\big([\varepsilon,T]; B^\infty L^p(\Om;\dot H^\alpha(\O))\big) , $$
which holds for arbitrary $\varepsilon\in(0,T)$. Moreover, the following quantitative estimates hold:
\begin{align}
\label{u-L2-estimate}
\|u(t)\|_{L^p(\Omega;L^2(\mathcal{O}))}
&\le
C(1+t^{\frac{\beta}{2}}) \hspace{5pt} \mbox{for}\,\,\, t\in(0,T] , 
\\
\label{B-Inf-estimate}
%{\b u(t)\in B^{\infty}L^p(\Om;\dot{H}^{\alpha}(\O))\,\,\, \mbox{and}}\,\,\,
\|u(t)\|_{B^{\infty}L^p(\Om;\dot{H}^{\alpha}(\O))}
&\le
Ct^{-\frac{\alpha-\beta}{2}} \hspace{15pt} \mbox{for}\,\,\, t\in(0,T] . 
\end{align}
\end{proposition}

\begin{remark}\label{RemarkProP}
{ 
For the trace-class noise, since \eqref{classical-alpha=1} holds with the classical Sobolev space, it follows that estimate \eqref{B-Inf-estimate} could be replaced by the following stronger result (with the classical Sobolev space for $\alpha=1$):
\begin{align*}
%{\b u(t)\in B^{\infty}L^p(\Om;\dot{H}^{\alpha}(\O))\,\,\, \mbox{and}}\,\,\,
\|u(t)\|_{L^p(\Om;\dot{H}^{1}(\O))}
&\le
Ct^{-\frac{1-\beta}{2}} \hspace{15pt} \mbox{for}\,\,\, t\in(0,T] . 
\end{align*}
}
\end{remark}

\subsection{The numerical method}\label{space-semi}

%In this section, we present error analysis for a fully discrete Galerkin method with a modified exponential Euler method in time and spectral method in space. 

Let $0=t_0<t_1<\cdots<t_N=T$ be a partition of the time interval $[0,T]$ with stepsizes $\tau_n=t_n-t_{n-1}\sim t_n^{\gamma} \tau$ {for $n=1,2,\cdots,N$}. 
The variable stepsizes defined in this way {have} the following properties: 
\begin{enumerate}
\item
$\tau_n\sim \tau_{n-1}$ for two consecutive stepsizes. 

\item 
$\tau_1=\tau^{\frac{1}{1-\gamma}}$. Hence, the starting stepsize is much smaller than the maximal stepsize. This can resolve the solution's singularity at $t=0$. 

\item
The total number of time levels is $O(T/\tau)$. Hence, the total computational cost is equivalent to using a uniform stepsize $\tau$. 
\end{enumerate}
This type of stepsizes was shown to be able to resolve the singularity at $t=0$ for semilinear parabolic equations with nonsmooth initial data; see \cite{Li-Ma-2020}. 

By using the variable stepsizes shown above, we consider the following modified exponential Euler method in time: 
%
%\begin{align}\label{spde-exp-1}
%u^n
%=&e^{-\tau_n A}u^{n-1}+\int_{t_{n-1}}^{t_n} e^{-(t_n-s)A}f(u^{n-1})\d s
%+ \int_{t_{n-1}}^{t_n} e^{-(t_n-s)A} \d W(s)  \quad \mbox{for}\,\,\, n\ge 1 , 
%\end{align}
\begin{align}\label{spde-exp-1}
\begin{aligned}
u^1
=&
e^{-\tau_1A}u^0,\\
u^n
=&e^{-\tau_n A}u^{n-1}+\int_{t_{n-1}}^{t_n} e^{-(t_n-s)A}f(u^{n-1})\d s
+ \int_{t_{n-1}}^{t_n} e^{-(t_n-s)A} \d W(s)  \quad \mbox{for}\,\,\, n\ge 2,
\end{aligned}  
\end{align}
where we assume that $\int_{t_{n-1}}^{t_n} e^{-(t_n-s)A} \d W(s)$ can be computed sufficiently accurately. 
This can be done by a spectral method in space with sufficiently many terms, as shown below. 
The semidiscrete scheme in \eqref{spde-exp-1} differs from the {\it exponential Euler method} only at the first  time level, where we drop the nonlinear term and the noise term. This is because that the nonlinear term $f(u^0)$ is generally not well defined for a nonsmooth initial value $u^0\in \dot H^{\beta}(\O)$.

Let $P_M$ be the $L^2$-orthogonal projection onto $S_M = \mbox{span}\{\phi_1,\dots,\phi_{M^d}\}$, defined by 
\begin{align*}
P_Mf=\sum_{k=1}^{M^d} f_k \phi_k \quad \mbox{for}\,\,\,\, f=\sum_{k=1}^\infty f_k \phi_k\in L^2(\O) .
\end{align*}
On a general bounded domain $\O$, we consider the following fully discrete spectral Galerkin method: 
\begin{subequations}\label{FD-method-1}
\begin{align}
U_M^1
&=
e^{-\tau_1A}U_M^0\quad\quad \mbox{with}\;\;U_M^0 = P_M u^0, \label{fullyUM1}
\\
U_M^n 
&=
e^{-\tau_n A}U_M^{n-1}+\int_{t_{n-1}}^{t_n}e^{-(t_n-s)A}P_Mf(U_M^{n-1})ds+\int_{t_{n-1}}^{t_n}e^{-(t_n-s)A}P_MdW(s) 
\quad\mbox{for}\,\,\, n\ge 2 . \label{fullyUMn}
\end{align}
\end{subequations}

In the one-dimensional case, e.g., $\O=(0,1)$, we can change to consider the following fast method which utilizes trigonometric interpolation and fast Fourier transform (FFT). 
%In this case, it is known that any function $f\in L^2(\Omega)$ can be expanded into the Fourier sine series, i.e.,  
%%\begin{align}
%%V=\sum_{j=1}^\infty V_{j} \phi_j . 
%%\end{align}
%\begin{align}
%f = \sum_{j=1}^\infty f_{j} \sin(j\pi x). 
%\end{align} 
Let $P_M$ and $I_M$ be the $L^2$-orthogonal projection and trigonometric interpolation operator (defined below) onto the finite dimensional space 
$$
S_M= \Big\{ \sum_{j=1}^M f_{j} \sin(j\pi x) :  
%\,\,\,\mbox{with}\,\,\,
f_{j}\in \R  \Big\} .
$$
Then we can consider the Fourier sine collocation method: 
\begin{subequations}\label{FD-method-2}
\begin{align}
U_M^1
&=
e^{-\tau_1A}U_M^0\quad\quad \mbox{with}\;\;U_M^0 = P_M u^0,%\label{fully1}
\\
U_M^n 
&=
e^{-\tau_n A}U_M^{n-1}+\int_{t_{n-1}}^{t_n}e^{-(t_n-s)A}I_Mf(U_M^{n-1})ds+\int_{t_{n-1}}^{t_n}e^{-(t_n-s)A}P_MdW(s) %\\
%&=
% e^{-(t_n-\tau_1) A}U_M^{1}
% +
%\sum_{j=2}^n \int_{t_{j-1}}^{t_j}e^{-(t_n-s)A}I_Mf(U_M^{j-1})ds
% +
% \int_{t_1}^{t_n}e^{-(t_n-s)A}P_MdW(s) \\
%&\hspace{300pt} 
\quad\mbox{for}\,\,\, n\ge 2 , %\notag
\label{FD-method-22}
\end{align}
\end{subequations} 
which only requires computing the trigonometric interpolation of the nonlinear function { $I_Mf(U_M^{n-1}):=I_M^*[f(U_M^{n-1})-f(0)]+f(0)P_M1$ instead of the $L^2$ projection $P_Mf(U_M^{n-1})$, where $I_M^*$ denotes the standard trigonometric interpolation operator onto $S_M$. The definition of $I_M$ guarantees that the standard trigonometric sine interpolation operator $I_M^*$ only acts on a function in $\dot H^1$ (satisfying the zero boundary condition) and therefore has optimal-order convergence; see the discussion below \eqref{v-IMv}.} The evaluation of the trigonometric interpolation $I_Mf(U_M^{n-1})$ could be done with $O(M\log M)$ operations at every time level by using FFT. This fast algorithm using trigonometric interpolation and FFT can also be extended to $d$-dimensional rectangular domains under the homogeneous Dirichlet boundary condition.
\begin{theorem}\label{main-theorem-0}
Let $\O$ be a bounded domain in $\R^d$ with $d\ge 1$.
Under %Assumptions (1)--(4),
Assumption \ref{thm-assumption}, 
 %{\b\ref{4-initial}--\ref{3-noise},}
%(1)--(4), 
by choosing $\gamma$ satisfying the following condition ({$\gamma$ is the constant from the relation $\tau_n\sim t_n^{\gamma} \tau$}):%({in the variable stepsize $\tau_n\sim t_n^{\gamma} \tau$}):
\be\label{gamma-condition0} 
\max\Big\{\frac12,1-\frac{1+\beta}{\alpha}\Big\}<\gamma<1,
\en
the numerical solution given by the fully discrete spectral Galerkin method \eqref{FD-method-1} has the following error bound: 
\begin{align*}
\|U_M^n-u(t_n)\|_{L^2(\Omega;L^2(\mathcal{O}))}
\le
C\tau^{\alpha} + CM^{-\alpha}t_n^{-\frac{\alpha-\beta}{2}} .
\end{align*}
\end{theorem}
\begin{theorem}\label{main-theorem}
Let $d=1$ and $\O=(0,1)$. 
%Under {Assumption \ref{thm-assumption}},%{\b Assumptions \ref{4-initial}--
Under Assumption \ref{thm-assumption}, 
%Assumptions (1)--(4),
%\ref{3-noise},}
%(1)--(4), 
by choosing $\gamma$ satisfying the following condition  ({$\gamma$ is the constant from the relation $\tau_n\sim t_n^{\gamma} \tau$}):%(in the variable stepsize $\tau_n\sim t_n^{\gamma} \tau$):
\be\label{gamma-condition1} 
\max\Big\{\frac12,1-\frac{1+\beta}{\alpha}\Big\} < \gamma<1,
\en
the numerical solution given by the fully discrete Fourier sine collocation method \eqref{FD-method-2} has the following error bound: 
\begin{align*}
\|U_M^n-u(t_n)\|_{L^2(\Omega;L^2(\mathcal{O}))}
\le
C\tau^{\alpha} + CM^{-\alpha}t_n^{-\frac{\alpha-\beta}{2}} .
\end{align*}
\end{theorem}

Since the spatial discretizations in \eqref{FD-method-1} and \eqref{FD-method-2} are both by spectral methods, the proofs for Theorems \ref{main-theorem-0} and \ref{main-theorem} are similar. Therefore, 
we present the proof for Theorem \ref{main-theorem} in the next section and omit the detailed proof for Theorem \ref{main-theorem-0}.

\section{Proof of Theorem \ref{main-theorem}}\label{proof_of_main}

%The proof of Theorem \ref{main-theorem} is divided into {\b five subsections}. 
%To prove the %sharp-order convergence
%{\b convergence order} {\b for the case of 
%one-dimensional space-time white noise,
%we present the real interpolation properties of the Besov space 
%$B^q L^p(\Omega;\dot H^\alpha(\O))$ %. These properties are used
% in Section \ref{section:interpolation}.}
% The space-time white noise 
%satisfies the second condition in Assumption (3) with $\alpha=\frac12$ only for the Besov norm $\|\cdot\|_{B^\infty L^2(\Omega;\dot H^\alpha(\O))}$,  
%{\b which is not valid
%%It is not valid with $\alpha=\frac12$
%when the Besov norm replaced by the Sobolev norm $\|\cdot\|_{L^2(\Omega;\dot H^\alpha(\O))}$.
% In Section \ref{section:error} we show the compositions of the error equation clearly and finish the estimates for its first two simple terms. 
% And then estimates for the defect (consistency error) and the noise term are provided in Sections \ref{section:defect} and \ref{section:noise errors}, respectively. 
% A completion subsection 
% followed by a completion subsection
% in the last subsection.% follow\Om;ing.
%% followed by a completion part in Section \ref{section:completion}.
%}

The proof of Theorem \ref{main-theorem} is divided into five subsections. 
In Section \ref{section:interpolation}, we present the real interpolation properties of the stochastic Besov spaces 
$B^q L^p(\Omega;\dot H^s(\O))$, $s\in[0,2]$. These properties are used to prove the sharp convergence order in the case of one-dimensional space-time white noise, which satisfies the second condition in %Assumption (3)
Assumption \ref{thm-assumption} (3) with $\alpha=\frac12$ for the stochastic Besov norm $\|\cdot\|_{B^\infty L^2(\Omega;\dot H^\alpha(\O))}$, but not for the  Sobolev norm $\|\cdot\|_{L^2(\Omega;\dot H^\alpha(\O))}$. 
The error estimates are presented in Sections \ref{section:error}--\ref{section:noise errors}.
%{\b The compositions of the error equation and estimates for its first two simple items are clearly shown in Section \ref{section:error}. While estimates for the defect (consistency error) and the noise term are provided in Sections \ref{section:defect} and \ref{section:noise errors}, respectively. And it follows a completion part in the last subsection.}

%In Section \ref{section:error} we show the compositions of the error equation clearly and finish the estimates for its first two simple terms. 
% And then estimates for the defect (consistency error) and the noise term are provided in Sections \ref{section:defect} and \ref{section:noise errors}, respectively. 
% A completion subsection 
% followed by a completion subsection
% in the last subsection.
% 
%The estimates for the defect (consistency error) and the error are presented in Sections \ref{section:defect} and \ref{section:error}, respectively. 

\subsection{Real interpolation results}
\label{section:interpolation}

The $K$-functional and $J$-functional are defined as 
\begin{align*}
K(t,f;X_0, X_1) &=\inf_{f=f_0+f_1}\|f_0\|_{X_0}+t\|f_1\|_{X_1} &&
\forall\, f\in X_0+ X_1,\,\,\, \forall\, t>0 ,\\
J(t,f;X_0, X_1) &=\max\{\|f\|_{X_0},t\|f\|_{X_1}\} 
&&\forall f\in X_0\cap X_1,\,\,\, \forall\, t>0 . 
\end{align*}

\begin{definition}[Discrete $K_{\theta,q}$-functor{\cite[Page 41, Lemma 3.1.3]{Bergh1976}}]\label{dis-K-functor} 
Let $0< \theta <1 $, $1\le q\le \infty$ and let $(X_0,X_1)$ be a compatible couple. 
The interpolation space $(X_0,X_1)_{\theta,q;K}$ consists of functions $f\in X_0+X_1$ such that 
$\|f\|_{(X_0,X_1)_{\theta,q;K}}<\infty$, where 
$$
\|f\|_{(X_0,X_1)_{\theta,q;K}} : = 
\left\{
\begin{aligned}
&\displaystyle 
\Big[\sum_{j\in\Z}\Big|a^{-j \theta}K(a^{j},f; X_0, X_1)\Big|^q\Big]^{1/q}, 
&&  1\le q<\infty,\\
&\displaystyle
\sup_{j\in\Z} a^{-j \theta}K(a^{j},f; X_0, X_1),&& q=\infty,
\end{aligned}
\right. 
$$
in which $a$ is any fixed positive constant.
\end{definition}
\begin{definition}[Discrete $J_{\theta,q}$-functor{\cite[Page 43, Lemma 3.2.3]{Bergh1976}}]\label{dis-J-functor}
Let $0< \theta <1 $, $1\le q\le \infty$ and let $(X_0,X_1)$ be a compatible couple. 
The interpolation space $(X_0,X_1)_{\theta,q;J}$ consists of function $f\in X_0+X_1$ such that $\|f\|_{(X_0,X_1)_{\theta,q;J}}<\infty$, where % with the following norm:  
%$\|f\|_{(X_0,X_1)_{\theta,q;J}}<\infty$ is defined as the infimum the norm 
$$
\|f\|_{(X_0,X_1)_{\theta,q;J}}
:= 
\left\{
\begin{aligned}
&\displaystyle
\inf_{f=\sum_{j}f_{j} }\Big(\sum_{j\in\Z}\Big|a^{-j \theta}J(a^{j},f_{j};X_0, X_1)\Big|^q\Big)^{\frac1q}, &&  1\le q<\infty,\\ 
&\displaystyle
\inf_{f=\sum_{j}f_{j} }\Big(\sup_{j\in\Z}a^{-j \theta}J(a^{j},f_{j}; X_0, X_1)\Big), && q=\infty,
\end{aligned}
\right. 
$$
in which $a$ is any fixed positive constant, and the infimum extends over all possible decompositions 
%Then $f\in (X_0,X_1)_{\theta,q;J}$ if and only if there exists $f_{j}\in X_0\cap X_1$ for all $j \in \mathbb{Z}$ such that
\be\label{f-fnu}
f=\sum_{j\in\Z}f_{j}  
\quad\mbox{with\, $f_{j}\in X_0\cap X_1$ and convergence in\, $X_0+X_1$} . 
\en
\end{definition}

It is known that $(X_0,X_1)_{\theta,q;K}$ and $(X_0,X_1)_{\theta,q;J}$ are the same vector space with equivalent norms; see \cite[Page 44, Theorem 3.3.1]{Bergh1976}). 
For simplicity, we denote by $(X_0,X_1)_{\theta,q}$ the common vector space of $(X_0,X_1)_{\theta,q;K}$ and $(X_0,X_1)_{\theta,q;J}$, with the norm {$\|\cdot\|_{(X_0,X_1)_{\theta,q}}$}. %{$\|\cdot\|_{\theta,q;K}$}. 

\begin{lemma}[\mbox{\cite[Page 301, Theorem 1.12]{Bennett-1988}}]\label{Operator-Inter}
Let $(X_0,X_1)$ and $(Y_0,Y_1)$ be compatible couples %and
and let $0<\theta<1$, $1\le q<\infty$ or $0\le \theta\le 1$, $q=\infty$. Let 
$T:X_0+X_1\rightarrow Y_0+Y_1$ be a linear operator such that $T$ maps
{$X_i$ to $Y_i$} %$X_j$ to $Y_j$ 
and
%with respect to $(X_0,X_1)$ and $(Y_0,Y_1)$ such that
$$
\|Tf \|_{Y_i}
\le M_i \|f\|_{X_i} \quad\forall\, f\in X_i,\,\, i=0,1 . 
$$
Then $T$ maps $(X_0,X_1)_{\theta,q}$ to $(Y_0,Y_1)_{\theta,q}$ and 
$$
\|Tf\|_{(Y_0,Y_1)_{\theta,q}}\le %\mbox{\color{red}(Is there a constant here?)} 
M_0^{1-\theta}M_1^{\theta}\|f\|_{(X_0,X_1)_{\theta,q}}
\quad\forall\, f\in (X_0,X_1)_{\theta,q}. 
$$
\end{lemma}
%$$
%\|u(t)\|_{B^{\infty}L^p(\Om;\dot{H}^{\zeta}(\O))}
%\le
%Ct^{-\frac{\zeta+\beta}{2}}\|u^0\|_{L^p(\Omega;\dot{H}^{-\beta}(\mathcal{O}))}+C.
%$$ 
%for all $\frac12\le\zeta\le 1$. 
%$$\sum_{k=2^{j-1}}^{2^j-1}\mu_k\lambda_k^{\alpha-1}$$

The main results of this subsection are presented in the following lemma. 
\begin{lemma}\label{real-interpolation2}
For all $1\le p,q\le \infty$ and $0<\theta<1$ there holds
\begin{align}\label{Inter-Besov1}
&B^qL^p(\Om;\dot{H}^s(\O))
=\big(L^p(\Omega;\dot{H}^{s_0}(\O)),L^p(\Omega;\dot{H}^{s_1}(\O))\big)_{\theta,q},\\
&\label{Inter-Besov}
{\big(B^{\infty}L^p(\Omega;\dot{H}^{s_0}(\O),B^{\infty}L^p(\Omega;\dot{H}^{s_1}(\O)\big)_{\theta,q}=B^qL^p(\Omega;\dot{H}^s(\O))} , 
\end{align}
where $-2\le s_0< s_1\le 2$ and $s=(1-\theta)s_0+\theta s_1$. %with $0<\theta<1$.
\end{lemma}

\begin{proof}
If $2^{j-1}\le k\le 2^j-1$, then $\lambda_k^s\sim k^{\frac{2s}{d}}\sim 2^{\frac{2js}{d}}$, which implies that 
\begin{align}\label{Pi-hs}
{\|\Pi_jf\|_{\dot{H}^s}\sim 2^{\frac{js}{d}}\|\Pi_jf\|_{L^2(\O)}.}
\end{align}
Hence 
for any $f\in L^p(\Omega;\dot{H}^{s_0})+L^p(\Omega;\dot{H}^{s_1})$, there exists a decomposition $f=f_0+f_1$ with $f_0\in L^p(\Omega;\dot{H}^{s_0})$ and $f_1\in L^p(\Omega;\dot{H}^{s_1})$. Since
\begin{align*}
&\|\Pi_{j}f\|_{L^p(\Omega;L^2(\O))} \\
&\le 
\|\Pi_{j}f_0\|_{L^p(\Omega;L^2(\O))}
+
\|\Pi_{j}f_1\|_{L^p(\Omega;L^2(\O))}\\
&= 
2^{-\frac{j s_0}{d}}\|2^{\frac{j s_0}{d}}\Pi_{j}f_0\|_{L^p(\Omega;L^2(\O))}
+
2^{-\frac{j s_1}{d}}\|2^{\frac{j s_1}{d}}\Pi_{j}f_1\|_{L^p(\Omega;L^2(\O))}\\
&\le
2^{-\frac{j s_0}{d}}\Big\|\big(\sum_{j=1}^{\infty}2^{\frac{2j s_0}{d}}\|\Pi_{j}f_0\|_{L^2(\O)}^2\big)^{\frac12}\Big\|_{L^p(\Omega)}
+ 
2^{-\frac{j s_1}{d}}\Big\|\big(\sum_{j=1}^{\infty}2^{\frac{2j s_1}{d}}\|\Pi_{j}f_1\|_{L^2(\O)}^2\big)^{\frac12}\Big\|_{L^p(\Omega)}\\
&\le 
C2^{-\frac{j s_0}{d}}\|f_0\|_{L^p(\Omega;\dot{H}^{s_0}(\O))}
+
C2^{-\frac{j s_1}{d}}\|f_1\|_{L^p(\Omega;\dot{H}^{s_1}(\O))},
\end{align*}
it follows that 
\begin{align*}
\|\Pi_{j}f\|_{L^p(\Om;L^2(\O))}
&\le 
C\inf_{f=f_0+f_1} 2^{-\frac{j s_0}{d}}\big(\|f_0\big\|_{L^p(\Omega;\dot{H}^{s_0}(\O))}
+
2^{\frac{j (s_0-s_1)}{d}}\|f_1\big\|_{L^p(\Omega;\dot{H}^{s_1}(\O))}\big)\\
&\le
C2^{-\frac{j s_0}{d}}K(2^{\frac{j(s_0-s_1)}{d}},f;X_0,X_1),
\end{align*}
where $X_0=L^p(\Om;\dot{H}^{s_0}(\O))$ and $X_1=L^p(\Om;\dot{H}^{s_1}(\O))$. This further implies that 
\begin{align*}
2^{\frac{j s}{d}}\|\Pi_{j}f\|_{L^p(\Om;L^2(\O))}
&\le 
C2^{\frac{j (s-s_0)}{d}}K(2^{\frac{j(s_0-s_1)}{d}},f;X_0,X_1)\\
&\le
C(2^{\frac{s_0-s_1}{d}})^{-j \theta}K((2^{\frac{(s_0-s_1)}{d}})^{j},f;X_0,X_1).
\end{align*}
By considering the $\ell^q$ norm of the inequality above with respect to $j$ and using Definition \ref{dis-K-functor}, 
we obtain 
\begin{align*}%\label{Result-real-1}
\|f\|_{B^qL^p(\Omega;\dot{H}^s(\O))}
%&=
%\Big(\sum_{\nu=-\infty}^{\infty}\big(2^{\frac{\nu s}{d}}\|\Pi_{\nu}f\|_{L^p(\Om;L^2(\O))}\big)^q\Big)^{\frac{1}{q}} \notag\\
&\le
C\|f\|_{(L^p(\Omega;\dot{H}^{s_0}(\O)),L^p(\Omega;\dot{H}^{s_1}(\O)))_{\theta,q}}
\quad\forall\, 1\le q\le \infty , 
\end{align*}
which means that 
\begin{align}\label{Result-real-1}
(L^p(\Omega;\dot{H}^{s_0}(\O)),L^p(\Omega;\dot{H}^{s_1}(\O)))_{\theta,q}\hookrightarrow B^qL^p(\Omega;\dot{H}^s(\O))
\quad\forall\, 1\le q\le \infty . 
\end{align}

Conversely, since $s_0<s<s_1$, if $f\in B^qL^p(\Omega;\dot{H}^s(\O))$ then 
$$f\in L^p(\Om;\dot{H}^{s_0}(\O))=L^p(\Omega;\dot{H}^{s_0}(\O))+ L^p(\Omega;\dot{H}^{s_1}(\O))$$ 
and  
\begin{align*}
\|f\|_{L^p(\Om;\dot{H}^{s_0}(\O))}
&\le
C\Big\|\big(\sum_{j=1}^{\infty}2^{\frac{2j s_0}{d}}\|\Pi_{j}f\|_{L^2(\O)}^2\big)^{\frac12}\Big\|_{L^p(\Omega)}\\
%&\le
%C\Big\|\sum_{j=1}^{\infty}2^{\frac{j s_0}{d}}\|\Pi_{j}f\|_{L^2(\O)}\Big\|_{L^p(\Omega)}\\
&\le
C\sum_{j=1}^{\infty}2^{\frac{j (s_0-s)}{d}} 2^{\frac{j s}{d}}\|\Pi_{j}f\|_{L^p(\Omega;L^2(\O))}\\
&\le
C\Big(\sum_{j=1}^{\infty}2^{\frac{j (s_0-s)}{d}q'} \Big)^{\frac{1}{q'}}
\Big(\sum_{j=1}^{\infty}\|\Pi_{j}f\|_{L^p(\Omega;\dot{H}^s(\O))}^q\Big)^{\frac{1}{q}}\\
&\le
C\|f\|_{B^qL^p(\Om;\dot{H}^s(\O))},
\end{align*}
{where $q'$ is the constant satisfying $\frac{1}{q}+\frac{1}{q'}=1$}.
Substituting this inequality into the expression $J(t,f;X_0, X_1) =\max\{\|f\|_{X_0},t\|f\|_{X_1}\} $ yields 
\begin{align*}
&2^{\frac{j(s-s_0)}{d}}J(2^{\frac{j(s_0-s_1)}{d}},\Pi_{j}f;L^p(\Om;\dot{H}^{s_0}(\O)),L^p(\Om;\dot{H}^{s_1}(\O)))\\
&=
2^{\frac{j(s-s_0)}{d}}\max\big\{\|\Pi_{j}f\|_{L^p(\Om;\dot{H}^{s_0}(\O))},2^{\frac{j(s_0-s_1)}{d}}\|\Pi_{j}f\|_{L^p(\Om;\dot{H}^{s_1}(\O))}\big\}\\
&\le 
C2^{\frac{j s}{d}}\|\Pi_{j}f\|_{L^p(\Om;L^2(\O))}\\
&\le 
C\|\Pi_{j}f\|_{L^p(\Om;\dot{H}^s(\O))}.
\end{align*}
In view of the inequality above and Definition \ref{dis-J-functor}, we have 
%where 
%$$
%J(t,f;X_0,X_1)=\max\big\{\|f_0\|_{X_0},t\|f_1\|_{X_1}\big\},
%$$
\begin{align*}%\label{Result-real-2}
\|f\|_{(L^p(\Om;\dot{H}^{s_0}(\O)),L^p(\Om;\dot{H}^{s_1}(\O)))_{\theta,q}}
\le
C\|f\|_{B^qL^p(\Omega;\dot{H}^s(\O))}
\quad\forall\, 1\le q\le \infty , 
\end{align*}
which means that 
\begin{align}\label{Result-real-2}
B^qL^p(\Omega;\dot{H}^s(\O))
\hookrightarrow (L^p(\Omega;\dot{H}^{s_0}(\O)),L^p(\Omega;\dot{H}^{s_1}(\O)))_{\theta,q} 
\quad\forall\, 1\le q\le \infty . 
\end{align}

The two results in \eqref{Result-real-1} and \eqref{Result-real-2} imply \eqref{Inter-Besov1}. 
Then \eqref{Inter-Besov} follows from \eqref{Inter-Besov1} as a result of the reiteration theorem in the real interpolation theory; see \cite[Page 50, Theorem 3.5.3]{Bergh1976}. 
%{\color{red}
%(the reiteration theorem can be found in \cite{Bergh1976})
%}
\end{proof}

{
The following inverse inequality will be often used in the error estimation: If $f\in S_M$ and $-2\le s_0\le s\le 2$ then 
\begin{align}\label{inverse1}
\|f\|_{\dot{H}^{s}(\O)}
=
%\|A^{\frac{\zeta}{2}}V_M\|_{L^2(\O)}^2
%=
\Big(\sum_{k=1}^{M}\lambda_k^{s} |(f,\phi_k)|^2 \Big)^{\frac12}
\le
C\Big(\lambda_M^{s-s_0}\sum_{k=1}^{M}\lambda_k^{s_0} |(f,\phi_k)|^2\Big)^{\frac12}
\le 
CM^{s-s_0}\|f\|_{\dot{H}^{s_0}(\O)} . 
\end{align}
Correspondingly, for a stochastic function $f\in L^p(\Omega;S_M)\hookrightarrow L^p(\Omega;\dot H^s(\O))$ the following result holds: 
\begin{align}\label{inverse2}
\|f\|_{L^p(\Omega;\dot{H}^{s}(\O))}
\le 
CM^{s-s_0}\|f\|_{L^p(\Omega;\dot{H}^{s_0}(\O))} 
\quad\mbox{for $-2\le s_0\le s\le 2$}. 
\end{align}
By choosing $-2\le s_2<s_0<s_1\le s\le 2 $ and consider the real interpolation between the two results, 
\begin{align*}
&\|f\|_{L^p(\Omega;\dot{H}^{s}(\O))}
\le 
CM^{s-s_1}\|f\|_{L^p(\Omega;\dot{H}^{s_1}(\O))} 
\quad\mbox{and}\quad
\|f\|_{L^p(\Omega;\dot{H}^{s}(\O))}
\le 
CM^{s-s_2}\|f\|_{L^p(\Omega;\dot{H}^{s_2}(\O))}  ,
\end{align*}
we obtain the following inequality for $f\in L^p(\Omega;S_M)\hookrightarrow L^p(\Omega;\dot H^s(\O))$: 
\begin{align}\label{inverse3}
&\|f\|_{L^p(\Omega;\dot{H}^{s}(\O))}
\le 
CM^{s-s_0}\|f\|_{B^\infty L^p(\Omega;\dot{H}^{s_0}(\O))} 
\quad\mbox{for $-2<s_0<s\le 2$}.
\end{align}
}

\subsection{The error equation}
\label{section:error}

%The mild solution of \eqref{SPDE1} satisfies the following equation: 
%\begin{align}\label{real_1}
%u(t_n) = e^{-t_nA}u^0+\sum_{i=1}^n\Big(\int_{t_{i-1}}^{{t_i}}e^{-(t_n-s)A}f(u(s)))ds+\int_{t_{i-1}}^{t_i}e^{-(t_n-s)A}dW(s)\Big) ,
%\end{align}
%where the remainder is given by 
%\begin{align*}
%{\mathcal{E}}^n_* 
%&=
%\sum_{i=1}^n\int_{t_{i-1}}^{t_i}e^{-(t_n-s)A}\Big(f(u(s))-f(e^{-(s-t_{i-1})A}u(t_{i-1}))\Big) \d s \quad\mbox{for}\,\,\, n\ge 1 . 
%\end{align*}

By iterating the second relation in \eqref{FD-method-2} for $n\ge 2$, 
the numerical solution in \eqref{FD-method-2} can be written as 
\begin{align}\label{expr-UMn}
U_M^n 
&=
 e^{-(t_n-\tau_1) A}U_M^{1}
 +
\sum_{j=2}^n \int_{t_{j-1}}^{t_j}e^{-(t_n-s)A}I_Mf(U_M^{j-1})ds
 +
 \int_{t_1}^{t_n}e^{-(t_n-s)A}P_MdW(s) .
\end{align}
The difference between \eqref{expr-UMn} and \eqref{mild-solution} yields the following expression for the error of the numerical solution: 
\begin{align}\label{error-repr}
U_M^n-u(t_n) 
&=
e^{-(t_n-\tau_1)A}U_M^1-e^{-(t_n-\tau_1)A}u(t_1) \notag\\
&\quad +
\sum_{j=2}^n\int_{t_{j-1}}^{t_j}\big(e^{-(t_n-s)A}I_Mf(U_M^{j-1})-e^{-(t_n-s)A}f(U_M^{j-1})\big)ds \notag\\
&\quad +
\sum_{j=2}^n\int_{t_{j-1}}^{t_j}\big(e^{-(t_n-s)A}f(U_M^{j-1})-e^{-(t_n-s)A}f(u(t_{j-1}))\big)ds \notag\\
&\quad+
\sum_{j=2}^n\int_{t_{j-1}}^{t_j}\big(e^{-(t_n-s)A}f(u(t_{j-1}))-e^{-(t_n-s)A}f(u(s))\big)ds \notag\\
&\quad +
\int_{t_1}^{t_n}\big(e^{-(t_n-s)A}P_M-e^{-(t_n-s)A}\big)dW(s) \notag\\
&=: 
\mathcal{E}_1^n+\mathcal{E}_2^n+\mathcal{E}_3^n+\mathcal{E}_4^n+\mathcal{E}_5^n , 
\end{align}
which is decomposed into five parts. 

By using the first relation in \eqref{FD-method-2} (when $n=1$), the first part on the right-hand side of \eqref{error-repr} can be further written as 
\begin{align} \label{estimate-E11}
\mathcal{E}_1^n
=&\ 
e^{- t_n A} P_Mu^0 - e^{- t_n A} u^0 \notag\\
&\ 
-e^{- (t_n-t_1) A} \int_0^{t_1} e^{-(t_1-s)A}f(u(s))\d s
-e^{- (t_n-t_1) A} \int_0^{t_1} e^{-(t_1-s)A}\d W(s) . 
\end{align} 
The first term in \eqref{estimate-E11} can be estimated by using the 
classical error estimates of spectral method for the heat equation, i.e.,
\begin{align*} 
\|e^{- t_n A} P_Mu^0 - e^{- t_n A} u^0\|_{L^2(\Om;L^2(\O))} 
\le &\ 
CM^{-\alpha}t_n^{-\frac{\alpha-\beta}{2}} \|u^0\|_{L^2(\Om;\dot{H}^{\beta}(\O))}
 . 
\end{align*} 
The second and third terms in \eqref{estimate-E11} can be estimated by using the regularity results in Proposition \ref{Regularity} and 
%Assumption (3)
Assumption \ref{thm-assumption} (3) 
 on the noise. Then we obtain 
\begin{align*} 
\|\mathcal{E}_1^n\|_{L^2(\Om;L^2(\O))} 
\le &\ 
CM^{-\alpha}t_n^{-\frac{\alpha-\beta}{2}} \|u^0\|_{\dot{H}^{\beta}(\O)}
+ \int_0^{\tau_1} C(1+s^{\frac{\beta}{2}}) \d s
+ C\tau_1^{\frac{\alpha}{2}}
\notag\\
\le&\
CM^{-\alpha}t_n^{-\frac{\alpha-\beta}{2}} 
+ C(\tau_1+\tau_1^{1+\frac{\beta}{2}})
+ C\tau_1^{\frac{\alpha}{2}} \notag\\
\le&\ 
CM^{-\alpha}t_n^{-\frac{\alpha-\beta}{2}}+C(\tau^{\frac{1}{1-\gamma}} + \tau^{\frac{1}{1-\gamma}(1+\frac{\beta}{2})} )
+C\tau^{\frac{1}{1-\gamma}\frac{\alpha}{2}} . 
\end{align*} 
If $\gamma\ge \max\big(\frac12, 1-(1+\frac{\beta}{2})\frac{1}{\alpha}\big)$, then the inequality above reduces to the following desired result: 
\begin{align} \label{estimate-E1} 
\|\mathcal{E}_1^n\|_{L^2(\Om;L^2(\O))} 
\le&\ 
CM^{-\alpha}t_n^{-\frac{\alpha-\beta}{2}} +C\tau^{\alpha} . 
\end{align} 

%Since $\alpha\in(0,1]$, 
%there exist 
%$-\infty<\epsilon_1<\epsilon_2<\infty$ such that 
%$\frac12<\alpha+\epsilon_1<1<\alpha+\epsilon_2$. 
{
The following estimates are known for the standard trigonometric interpolation $I_M^*$ and $L^2$ projection $P_M$:  
%{\color{red}(the book of Rainer Kress, "Integral equations", contains such error estimates for trigonometric interpolation)} 
\begin{align}\label{v-IMv}
\|v - I_M^* v\|_{L^2(\Omega;L^2(\O))}
&\le
CM^{-s} \|v\|_{L^2(\Omega;\dot H^s(\O))} \quad\mbox{for}\,\,\, v\in \dot H^s(\O), \quad \frac12<s\le 2 ,\\
\label{v-PMv}
\|v - P_M v\|_{L^2(\Omega;L^2(\O))}
&\le
CM^{-s} \|v\|_{L^2(\Omega;\dot H^s(\O))} \quad\mbox{for}\,\,\, v\in \dot H^s(\O), \quad 0\le s\le 2 , 
%\quad\mbox{for}\,\,\,
%\epsilon_1<\epsilon<\epsilon_2.
\end{align}
where the error estimates of trigonometric interpolation for periodic functions can be found in \cite[Page 209, Theorem 11.8]{kress-2014}; the error estimates of trigonometric sine interpolation for functions satisfying the Dirichlet boundary condition on $\O=(0,1)$ follow by extending the function to $[-1,1]$ as a periodic odd function. 
Since odd extension of a function preserves the continuity of the function and its first-order derivative, it follows that the odd extension maps $\dot H^s(0,1)$ to the periodic function space $H^s_{\rm per}[-1,1]$ for $s\in[0,2]$ (by considering the real interpolation between the two endpoint cases $s=0$ and $s=2$).
 
%\mbox{\color{red}this may not be right since $\alpha+\epsilon_1$ and 
%$\alpha+\epsilon_2$ could be both larger than $\alpha$}
%\mbox{\color{red} Was this derived by the real interpolation? I think it's just the classical result of spectral method}

For the function $g=f(U_M^{j-1})-f(0)$ which satisfies the zero boundary condition, the following error estimate holds: 
\begin{align*}
\|I_M^*g - g \|_{L^2(\Omega;L^2(\O))}
\le CM^{-1} \| g \|_{L^2(\Omega;\dot H^1(\O))} .
%\le CM^{-1} \| f(U_M^{j-1}) -f(0) \|_{L^2(\Omega;\dot H^1(\O))}. 
\end{align*}
Since $e^{-(t-s)A}$ commutes with the projection operator $P_M$, it follows that 
\begin{align*}
\|e^{-(t-s)A}(P_M1-1) \|_{L^2(\Omega;L^2(\O))}
&=\|P_M e^{-(t-s)A}1 - e^{-(t-s)A}1 \|_{L^2(\Omega;L^2(\O))} \\
&\le CM^{-1} \| e^{-(t-s)A}1  \|_{L^2(\Omega;\dot H^1(\O))} \\
&\le C(t-s)^{-\frac{1}{2}} M^{-1} \| 1 \|_{L^2(\Omega;L^2(\O))} \\
&\le C(t-s)^{-\frac{1}{2}} M^{-1} .
\end{align*}
Therefore, by using expression $I_Mf(U_M^{j-1}) = I_M^*g + f(0)P_M1$ and the triangle inequality, 
\begin{align*}
&\|e^{-(t-s)A}[I_Mf(U_M^{j-1}) - f(U_M^{j-1})]\|_{L^2(\Omega;L^2(\O))} \\
&\le \| e^{-(t-s)A} (I_M^*g - g) \|_{L^2(\Omega;L^2(\O))}
     + |f(0)| \| e^{-(t-s)A}(P_M1-1)  \|_{L^2(\Omega;L^2(\O))} \\ 
&\le CM^{-1} \| f(U^{j-1}_M)-f(0) \|_{L^2(\Omega;\dot H^1(\O))} 
     + C|f(0)| (t-s)^{-\frac{1}{2}} M^{-1} .
\end{align*}
\!\!}
Then, using \eqref{v-IMv}--\eqref{v-PMv} and the inverse inequalities in \eqref{inverse2}--\eqref{inverse3}, the second term on the right-hand side of \eqref{error-repr} can be estimated as follows for $\alpha\in(0,1)$: 
%Furthermore, by Lemma \ref{u-H-alpha} and Lemma \ref{A-delta-real-interpolation} {if $0\le\beta_0<2-\alpha$ we have}
\begin{align}\label{E2n-estimate-1}
&\|\mathcal{E}_2^n\|_{L^2(\Om;L^2(\O))} \notag\\
&\le%1
C\sum_{j=2}^n\int_{t_{j-1}}^{t_j}
M^{-1} \big( \|f(U^{j-1}_M)-f(0)\|_{L^2(\Om;\dot H^{1}(\O))}
 + |f(0)| (t-s)^{-\frac{1}{2}} \big) \d s \notag\\
&\le%2
{
C\sum_{j=2}^n\int_{t_{j-1}}^{t_j}M^{-1}
\big( \|U^{j-1}_M\|_{L^2(\Om;\dot H^{1}(\O))} + (t-s)^{-\frac{1}{2}} \big) \d s
} \notag\\
&\le%2.5
{C\sum_{j=2}^n\int_{t_{j-1}}^{t_j}M^{-1} \big(\|U^{j-1}_M-P_Mu(t_{j-1})\|_{L^2(\Om;\dot H^{1}(\O))}
+\|P_Mu(t_{j-1})\|_{L^2(\Om;\dot H^{1}(\O))} \big) \d s
  + CM^{-1} } \notag\\
  &\le%3
{C\sum_{j=2}^n\int_{t_{j-1}}^{t_j}\big(\|U^{j-1}_M-P_Mu(t_{j-1})\|_{L^2(\Om;L^2(\O))}+M^{-\alpha}\|P_Mu(t_{j-1})\|_{B^\infty L^2(\Om;\dot H^{\alpha}(\O))}\big) \d s + CM^{-1} }\notag\\
&\le%3
C\sum_{j=2}^n\int_{t_{j-1}}^{t_j}\big(\|U^{j-1}_M-P_Mu(t_{j-1})\|_{L^2(\Om;L^2(\O))}+M^{-\alpha}\|u(t_{j-1})\|_{B^\infty L^2(\Om;\dot H^{\alpha}(\O))}\big) \d s + CM^{-1}  ,
\end{align}
{ where we have used the inverse inequalities in \eqref{inverse2}--\eqref{inverse3} for $\alpha\in(0,1)$ and the stability of $P_M$ on the Besov space. The stochastic Besov norm $\|u(t_{j-1})\|_{B^\infty L^2(\Om;\dot H^{\alpha}(\O))}$ in the last inequality can be furthermore estimated by the regularity result in \eqref{B-Inf-estimate}. In the case $\alpha=1$, for trace class noise, we can simply replace the stochastic Besov norm $\|P_Mu(t_{j-1})\|_{B^\infty L^2(\Om;\dot H^{\alpha}(\O))}$ by the classical Sobolev norm $\|P_Mu(t_{j-1})\|_{L^2(\Om;\dot H^{\alpha}(\O))}$ and use the regularity result in Remark \ref{RemarkProP}. In both cases, we can furthermore obtain the following result:}
\begin{align}\label{E2n-estimate-2}
\|\mathcal{E}_2^n\|_{L^2(\Om;L^2(\O))} 
&\le
C\sum_{j=2}^n \tau_j \|U^{j-1}_M-P_Mu(t_{j-1})\|_{L^2(\Om;L^2(\O))}
+ C\sum_{j=2}^n \tau_j t_{j-1}^{-\frac{\alpha-\beta}{2}} M^{-\alpha} + {  CM^{-1}} \notag\\
&\le
C\sum_{j=2}^n \tau_j \|U^{j-1}_M-u(t_{j-1})\|_{L^2(\Om;L^2(\O))}+CM^{-\alpha} .
\end{align}

The third term on the right-hand side of \eqref{error-repr} can be estimated directly by using the Lipschitz continuity of $f$, i.e., 
\begin{align}\label{estimate-E3}
\|\mathcal{E}_3^n\|_{L^2(\Om;L^2(\O))}
&\le
C\sum_{j=2}^n\tau_j \|U_M^{j-1}-u(t_{j-1})\|_{L^2(\Om;L^2(\O))}.
\end{align}
The estimation for the fourth and fifth terms on the right-hand side of \eqref{error-repr} are the technical parts in the proof, and are presented in the following two subsections, respectively.

\subsection{Estimation of $\mathcal{E}_4^n$}
\label{section:defect}

\begin{lemma}\label{Lemma:E4n}
\it 
Under {Assumption \ref{thm-assumption}},
%s {\rm(1)--(4)}, 
%Assumptions {\b \ref{4-initial}--\ref{3-noise}}, %{\rm(1)--(4)}, 
the remainder ${\mathcal{E}}^n_{4}$ in \eqref{error-repr} has the following bound: 
\begin{align}\label{temporal-2}
\|{\mathcal{E}}^n_{4} \|_{L^2(\Omega;L^2(\mathcal{O}))}
\le C\tau^{\alpha} .
\end{align} 
\end{lemma}

\begin{proof}
The estimation of $\mathcal{E}_4^n$ is by introducing an intermediate term $f(e^{-(s-t_{j-1})A}u(t_{j-1}))$ between $f(u(s))$ and $f(u(t_{j-1})$. 
By this means, we decompose $\mathcal{E}_4^n$ into the following two parts: 
\begin{align}\label{E4n-2-parts}
\mathcal{E}_4^n 
=&\
 \sum_{j=2}^n\int_{t_{j-1}}^{t_j} e^{-(t_n-s)A}[f(u(s))-f(u(t_{j-1})]\d s \notag\\
=&\
\sum_{j=2}^n\int_{t_{j-1}}^{t_j} e^{-(t_n-s)A}[f(e^{-(s-t_{j-1})A}u(t_{j-1}))-f(u(t_{j-1})]\d s
\notag\\
&\ +
{\sum_{j=1}^n}\int_{t_{j-1}}^{t_j} e^{-(t_n-s)A}[f(u(s))-f(e^{-(s-t_{j-1})A}u(t_{j-1}))]\d s \notag\\
=&\!:  
\mathcal{E}_{4,1}^n+\mathcal{E}_{4,2}^n.
\end{align}
The two parts are estimated separately.

Since $A=-\Delta$, the first part on the right-hand side of \eqref{E4n-2-parts} can be furthermore written as 
\begin{align*}
\mathcal{E}_{4,1}^n
&=
{-\sum_{j=2}^n\int_{t_{j-1}}^{t_j} e^{-(t_n-s)A}\int_{0}^{s-t_{j-1}}f'(e^{-\delta A}u(t_{j-1}))\Delta(e^{-\delta A}u(t_{j-1}))\d\delta\d s}\\
&=
-\sum_{j=2}^n\int_{t_{j-1}}^{t_j}e^{-(t_n-s)A}\int_{0}^{s-t_{j-1}} \nabla\cdot\Big(f'(e^{-\delta A}u(t_{j-1}))\nabla\big(e^{-\delta A}u(t_{j-1})\big)\Big)\d\delta\d s\\
&\quad +
\sum_{j=2}^n\int_{t_{j-1}}^{t_j}e^{-(t_n-s)A} \int_{0}^{s-t_{j-1}} f''(e^{-\delta A}u(t_{j-1}))\Big|\nabla\big(e^{-\delta A}u(t_{j-1})\big)\Big|^2\d\delta\d s.
\end{align*}
Since $\dot{H}^1(\O)=H_0^1(\O)$, it follows that $\dot{H}^{-1}(\O)=H^{1}_0(\O)'$. As a result, the following result holds for all $\vec h\in {L^2(\O)^d}$: 
\begin{align*}
\|A^{-\frac12}\nabla\cdot \vec{h}\|_{L^2(\O)}
=
\|\nabla\cdot \vec{h}\|_{H^{-1}(\O)}
=
\sup_{g\in H_0^1(\O)} \frac{|(\nabla\cdot \vec{h},g)|}{\|g\|_{H_0^1(\O)}}
&=
\sup_{g\in H_0^1(\O)} \frac{|(\vec{h},\nabla g)|}{\|g\|_{H_0^1(\O)}} 
\le
C\|\vec{h}\|_{L^2(\O)} . 
\end{align*}
By using this result and 
%Assumption (2), 
item (2) in Assumption \ref{thm-assumption},
%{\color{red}Assumption \ref{thm-assumption} (2)},
we have 
\begin{align*}
&\|\mathcal{E}_{4,1}^n\|_{L^2(\Omega;L^2(\mathcal{O}))}\\
&\le%1
C\sum_{j=2}^n\int_{t_{j-1}}^{t_j} \hspace{-3pt} \int_{0}^{s-t_{j-1}} \hspace{-3pt} \Big\|A^{\frac12}e^{-(t_n-s)A}A^{-\frac12}\nabla\cdot\Big(f'(e^{-\delta A}u(t_{j-1}))\nabla\big(e^{-\delta A}u(t_{j-1})\big)\Big)\Big\|_{L^2(\Omega;L^2(\mathcal{O}))} \hspace{-23pt} \d\delta\d s\\
&\quad +
C\sum_{j=2}^n\int_{t_{j-1}}^{t_j}\int_{0}^{s-t_{j-1}} \Big\|A^{\frac{\eta}{2}}e^{-(t_n-s)A}A^{-\frac{\eta}{2}}f''(e^{-\delta A}u(t_{j-1}))\big|\nabla\big(e^{-\delta A}u(t_{j-1})\big)\big|^2\Big\|_{L^2(\Omega;L^2(\mathcal{O}))} \hspace{-23pt} \d\delta\d s\\
%&\le%2 
%C\sum_{j=2}^n\int_{t_{j-1}}^{t_j}\int_{0}^{s-t_{j-1}} (t_n-s)^{-\frac12}\Big\|A^{-\frac12}\nabla\cdot\Big(f'(e^{-\delta A}u(t_{j-1}))\nabla\big(e^{-\delta A}u(t_{j-1})\big)\Big)\Big\|_{L^2(\Om;L^2(\O))}\d \delta\d s\\
%&\quad +
%C\sum_{j=2}^n\int_{t_{j-1}}^{t_j}\int_0^{s-t_{j-1}}(t_n-s)^{-\frac{\eta}{2}}\big\|A^{-\frac{\eta}{2}}f''(e^{-\delta A}u(t_{j-1}))\big|\nabla\big(e^{-\delta A}u(t_{j-1})\big)\big|^2\big\|_{L^2(\Om;L^2(\O))}\d\delta \d s\\
&\le%3 
C\sum_{j=2}^n\int_{t_{j-1}}^{t_j}\int_{0}^{s-t_{j-1}} (t_n-s)^{-\frac12}\Big\|f'(e^{-\delta A}u(t_{j-1}))\nabla\big(e^{-\delta A}u(t_{j-1})\big)\Big\|_{L^2(\Om;L^2(\O))}\d \delta\d s\\
&\quad +
C\sum_{j=2}^n\int_{t_{j-1}}^{t_j}\int_0^{s-t_{j-1}}(t_n-s)^{-\frac{\eta}{2}}\big\|\nabla\big(e^{-\delta A}u(t_{j-1})\big)\big\|_{L^4(\Om;L^2(\O))}^2\d\delta \d s.
\end{align*}
The analyticity of the semigoup $e^{-t A}$ implies the following estimates for all $1\le p\le\infty$: 
\begin{align*}
\|e^{-\delta A}u(t)\|_{L^p(\Om;\dot{H}^1(\O))}
&\le 
C\delta^{-\frac12}\|u(t)\|_{L^p(\Om;L^2(\O))},\\
\|e^{-\delta A}u(t)\|_{L^p(\Om;\dot{H}^1(\O))}
&\le 
C\|u(t)\|_{L^p(\Om;\dot{H}^1(\O))}.
\end{align*}
By using Lemma \ref{Operator-Inter}, Lemma \ref{real-interpolation2} and the result \eqref{B-Inf-estimate} in Proposition \ref{Regularity}, we obtain
\begin{align*}
\|e^{-\delta A}u(t)\|_{L^p(\Om;\dot{H}^1(\O))}
&\le%1
C\delta^{\frac{\alpha-1}{2}}\|u(t)\|_{(L^p(\Om;L^2(\O)),L^p(\Om;\dot{H}^1(\O)))_{\alpha,\infty}}\\
&\le%2
C\delta^{\frac{\alpha-1}{2}}\|u(t)\|_{B^{\infty}L^p(\Om;\dot{H}^{\alpha}(\O))}\\
&\le%3
C\delta^{\frac{\alpha-1}{2}}t^{-\frac{\alpha-\beta}{2}}\|u^0\|_{L^p(\Om;\dot{H}^{\beta}(\O))}
\end{align*}
for $p\in\{2,4\}$ and  $\beta\in(-1,\alpha]$.
Therefore, by using the inequality above, we have  
\begin{align}\label{estimate-E41}
\|\mathcal{E}_{4,1}^n\|_{L^2(\Omega;L^2(\mathcal{O}))} 
&\le%1
C\sum_{j=2}^n\int_{t_{j-1}}^{t_j}\int_{0}^{s-t_{j-1}} (t_n-s)^{-\frac12}\delta^{\frac{\alpha-1}{2}}t_{j-1}^{-\frac{\alpha-\beta}{2}}\d \delta\d s\|u^0\|_{ L^2(\Omega;\dot{H}^{\beta}(\O))} \notag\\
&\quad +
C\sum_{j=2}^n\int_{t_{j-1}}^{t_j}\int_0^{s-t_{j-1}}(t_n-s)^{-\frac{\eta}{2}}\delta^{\alpha-1}t_{j-1}^{-(\alpha-\beta)}\d\delta \d s\|u^0\|^2_{ L^4(\Omega;\dot{H}^{\beta}(\O))} \notag\\
&\le%2
C\sum_{j=2}^n\int_{t_{j-1}}^{t_j}\Big( (t_n-s)^{-\frac12}\tau_j^{\frac{\alpha+1}{2}}t_{j-1}^{-\frac{\alpha-\beta}{2}}
+(t_n-s)^{-\frac{\eta}{2}}\tau_j^{\alpha}t_{j-1}^{-(\alpha-\beta)}\Big) \d s \notag\\
%&\le%3
%C\sum_{j=2}^n\int_{t_{j-1}}^{t_j} \Big((t_n-s)^{-\frac12}\tau^{\frac{\alpha+1}{2}}t_{j-1}^{\frac{\gamma(\alpha+1)-\alpha+\beta}{2}}
%+(t_n-s)^{-\frac{\eta}{2}}\tau^{\alpha} t_{j-1}^{\alpha\gamma-\alpha+\beta}\Big) \d s \\
&\le%4
C\sum_{j=2}^n\int_{t_{j-1}}^{t_j}
\Big( (t_n-s)^{-\frac12}\tau^{\frac{\alpha+1}{2}}s^{\frac{\gamma(\alpha+1)-\alpha+\beta}{2}}
+(t_n-s)^{-\frac{\eta}{2}}\tau^{\alpha} s^{\alpha\gamma-\alpha+\beta}\Big) \d s ,
%\\
%&\le 
%C\tau^{\alpha}.
\end{align}
where the second to last inequality has used $\tau_j\sim t_j^{\gamma}\tau$.
Since condition \eqref{gamma-condition1} implies  
$$
\frac{\gamma(\alpha+1)-\alpha+\beta}{2} > -1 
\quad\mbox{and}\quad
\alpha\gamma-\alpha+\beta>-1 
$$
it follows that 
\begin{align}\label{E41-estimate}
\|\mathcal{E}_{4,1}^n\|_{L^2(\Omega;L^2(\mathcal{O}))}
\le C\tau^{\alpha} .
\end{align}

The second part on the right-hand side of \eqref{E4n-2-parts} can be further decomposed into three parts by using Taylor's expansion, i.e.,  
\begin{align}\label{def-En=star}
\mathcal{E}_{4,2}^n
&=
\sum_{i=1}^n\int_{t_{i-1}}^{t_i}e^{-(t_n-s)A}f'(e^{-(s-t_{i-1})A}u(t_{i-1}))\big(u(s)-e^{-(s-t_{i-1})A}u(t_{i-1})\big)\d s \notag\\
&\quad + \sum_{i=1}^n\int_{t_{i-1}}^{t_i} e^{-(t_n-s)A}f''(\xi_s)\big|u(s)-e^{-(s-t_{i-1})A}u(t_{i-1})\big|^2\d s\notag\\
&=
\sum_{i=1}^n\int_{t_{i-1}}^{t_i}e^{-(t_n-s)A}f'(e^{-(s-t_{i-1})A}u(t_{i-1})) \int_{t_{i-1}}^se^{-(s-\delta)A}f(u(\delta))\d\delta \d s \notag\\
&\quad +
\sum_{i=1}^n\int_{t_{i-1}}^{t_i}e^{-(t_n-s)A}f'(e^{-(s-t_{i-1})A}u(t_{i-1})) \int_{t_{i-1}}^s e^{-(s-\delta)A}\d W(\delta) \d s \notag\\
&\quad +
 \sum_{i=1}^n\int_{t_{i-1}}^{t_i} A^{\frac{\eta}{2}}e^{-(t_n-s)A}A^{-\frac{\eta}{2}}
 \Big( f''(\xi_s)\big|u(s)-e^{-(s-t_{i-1})A}u(t_{i-1})\big|^2 \Big) \d s \notag\\
&=
{\mathcal{E}}^n_{*,1} + {\mathcal{E}}^n_{*,2} + {\mathcal{E}}^n_{*,3} . 
\end{align}
%In the following, we estimate the three parts separately. 
The first part on the right-hand side of \eqref{def-En=star} can be estimated directly by using the regularity results in Proposition \ref{Regularity}, i.e., 
\begin{align*}
\|{\mathcal{E}}^n_{*,1} \|_{L^2(\Omega;L^2(\mathcal{O}))}
&\le 
C\sum_{i=1}^n\int_{t_{i-1}}^{t_i}\int_{t_{i-1}}^s \|f(u(\delta))\|_{L^2(\Omega;L^2(\mathcal{O}))} \d\delta \d s\\
&\le 
C\sum_{i=1}^n\int_{t_{i-1}}^{t_i}\int_{t_{i-1}}^s(1+\|u(\delta)\|_{L^2(\Omega;L^2(\mathcal{O}))})\d\delta \d s\\
&\le
C\sum_{i=1}^n\int_{t_{i-1}}^{t_i}\int_{t_{i-1}}^s (1+\delta^{\frac{\beta}{2}})\d\delta\d s\\
&\le
C\sum_{i=1}^n\Big(\tau_i+t_i^{1+\frac{\beta}{2}}-t_{i-1}^{1+\frac{\beta}{2}}\Big)\tau +C\tau\\
&\le
C\tau .
\end{align*}
The second part on the right-hand side of \eqref{def-En=star} can be estimated by using %a consequence of 
{It\^o's isometry}
%{a Burkholder–Davis–Gundy-type inequality \cite[Lemma 4.2]{Wang-2017}} 
and
%Assumption (3),
item (3) in Assumption \ref{thm-assumption},
as shown in \eqref{noise-equiv} (see Appendix), i.e., 
\begin{align*}
&\|{\mathcal{E}}^n_{*,2} \|_{L^2(\Omega;L^2(\mathcal{O}))}^2 \\
&\le C 
\sum_{i=1}^n\Big\|\int_{t_{i-1}}^{t_i}e^{-(t_n-s)A}f'(e^{-(s-t_{i-1})A}u(t_{i-1}))\int_{t_{i-1}}^s e^{-(s-\delta)A}\d W(\delta)\d s\Big\|_{L^2(\Omega;L^2(\mathcal{O}))}^2\\
&\le C 
\sum_{i=1}^n\Big(\int_{t_{i-1}}^{t_i}\Big\|\int_{t_{i-1}}^s e^{-(s-\delta)A}\d W(\delta)\Big\|_{L^2(\Omega;L^2(\mathcal{O}))}\d s\Big)^2\\
&\le C 
\sum_{i=1}^n\Big(\int_{t_{i-1}}^{t_i}(s-t_{i-1})^{\frac{\alpha}{2}}\d s\Big)^2
\quad\mbox{(here \eqref{noise-equiv} in Appendix is used)} \\
%&\lesssim
%\sum_{i=1}^n\tau_i^{\frac32}\Big(\int_{t_{i-1}}^{t_i}\Big\|\int_{t_{i-1}}^s e^{(s-\delta)A}dW(\delta)\Big\|_{L^2(\Omega;X)}^4ds\Big)^{\frac12}\\
%&\lesssim 
%\sum_{i=1}^n\tau_i^{\frac32}\Big(\int_{t_{i-1}}^{t_i}\Big\|\int_{t_{i-1}}^s e^{(s-\delta)A}dW(\delta)\Big\|_{L^4(\Omega;X)}^4ds\Big)^{\frac12}\\
%&\lesssim
%\sum_{i=1}^n\tau_i^{\frac32}\Big(\int_{t_{i-1}}^{t_i}(s-t_{i-1})^{2\alpha} ds\Big)^{\frac12}\\
%&\lesssim
%\sum_{i=1}^n\tau_i^2\tau_i^{\alpha}\\
&\le C
\tau^{1+\alpha} . 
\end{align*}
%where the second to last inequality uses \eqref{noise-equiv} in Appendix. 
This proves the following result: 
$$
\|{\mathcal{E}}^n_{*,2} \|_{L^2(\Omega;L^2(\mathcal{O}))}\le C\tau^{\frac{1+\alpha}{2}}.
$$

The third part on the right-hand side of \eqref{def-En=star}  can be estimated by using the following identity for $s\in[t_{i-1},t_i]$: 
\begin{align*}%\label{real_22}
u(s) = e^{-(s-t_{i-1})A}u(t_{i-1}) 
+\int_{t_{i-1}}^{s}e^{-(s-t)A}f(u(t))\d t +\int_{t_{i-1}}^{s}e^{-(s-t)A}\d W(t)  ,
\end{align*}
which implies that  
\begin{align*}
&\|u(s)-e^{-(s-t_{i-1})A}u(t_{i-1})\|_{L^4(\Omega;L^2(\mathcal{O}))}\\
&\le%1 
\Big\|\int_{t_{i-1}}^se^{-(s-t)A}f(u(t))\d t\Big\|_{L^4(\Omega;L^2(\mathcal{O}))}
+\Big\|\int_{t_{i-1}}^se^{-(s-t)A}\d W(t)\Big\|_{L^4(\Omega;L^2(\mathcal{O}))}\\
&\le%2 
C\int_{t_{i-1}}^s{\color{black}(1+\|u(t)\|_{L^4(\Omega;L^2(\mathcal{O}))})}\d t
+C\Big\|\int_{0}^{s-t_{i-1}}e^{-(s-t_{i-1}-t)A}\d W(t)\Big\|_{L^4(\Omega;L^2(\mathcal{O}))}
\\
&\le%3 
C\int_{t_{i-1}}^s (1+ t^{\frac{\beta}{2}}) \d t
+
C\Big\|\int_{0}^{s-t_{i-1}}e^{-(s-t_{i-1}-t)A}\d W(t)\Big\|_{L^4(\Omega;L^2(\mathcal{O}))}
\\
%&\le%4
%C(s-t_{i-1})+C(s^{1-\frac{\beta_0}{2}}-t_{i-1}^{1-\frac{\beta_0}{2}}) +{C(s-t_{i-1})^{\frac{\alpha}{2}} }\\
%&\le%5
%C\tau+C\tau^{1-\frac{\beta_0}{2}} +C\tau^{\frac{\alpha}{2}}\\
&\le
C\tau^{\frac{\alpha}{2}} , 
\end{align*}
where the second to last inequality uses the regularity results in Proposition \ref{Regularity}. 
Then, by using %Assumption (2) 
item (2) in Assumption \ref{thm-assumption} 
on the nonlinearity, we have 
\begin{align*}
&\|{\mathcal{E}}^n_{*,3} \|_{L^2(\Omega;L^2(\mathcal{O}))}\\
&\le
{\small
C\sum_{i=1}^n\int_{t_{i-1}}^{t_i}(t_n-s)^{-\frac{\eta}{2}} \Big\|A^{-\frac{\eta}{2}}\Big( f''(\xi_s)\big|u(s)-e^{-(s-t_{i-1})A}u(t_{i-1})\big|^2 \Big) \Big\|_{L^2(\Omega;L^2(\mathcal{O}))}\d s}\\
&\le
C\sum_{i=1}^n\int_{t_{i-1}}^{t_i}(t_n-s)^{-\frac{\eta}{2}}\|u(s)-e^{-(s-t_{i-1})A}u(t_{i-1})\|_{L^4(\Omega;L^2(\mathcal{O}))}^2\d s\\
&\le
C\sum_{i=1}^n\Big((t_n-t_{i-1})^{1-\frac{\eta}{2}}-(t_n-t_{i})^{1-\frac{\eta}{2}}\Big)\tau^{\alpha}\\
&\le
C\tau^{\alpha} .
\end{align*}
By substituting the estimates of $\|{\mathcal{E}}^n_{*,j}\|_{L^2(\Omega;L^2(\mathcal{O}))}$, $j=1,2,3,$ into \eqref{def-En=star}, we obtain
\begin{align}\label{E42-estimate}
\|{\mathcal{E}}^n_{4,2} \|_{L^2(\Omega;L^2(\mathcal{O}))}
%\le C\tau + C\tau^{\frac{1+\alpha}{2}} + C\tau^{\alpha}
\le C\tau^{\alpha} .
\end{align} 
Estimates \eqref{E41-estimate} and \eqref{E42-estimate} imply the desired result of 
Lemma \ref{Lemma:E4n}. 
\end{proof}

%Therefore
%\be\label{temporal-2}
%\begin{aligned}
%\|\mathcal{E}_n\|_{L^2(\Omega;L^2(\mathcal{O}))}
%&\le
% \|\widetilde{\mathcal{E}}_n\|_{L^2(\Omega;L^2(\mathcal{O}))}+\|\overline{\mathcal{E}}_n\|_{L^2(\Omega;L^2(\mathcal{O}))}\\
% &\le
%C\tau^{\alpha}\|u^0\|_{L^4(\Om;\dot{H}^{-\beta_0}(\O))}
%\end{aligned}
%\en
%under condition \eqref{gamma-condition2}.

\subsection{Estimation of $\mathcal{E}_5^n$}\label{section:noise errors}

The last term on the right-hand side of \eqref{error-repr} can be estimated as follows, by considering two different cases and using the real interpolation method between two Besov spaces. 

Case 1: $t_n>M^{-2}$.
On the one hand, for any $\delta\in(0,\alpha)$ the following estimate holds: 
\begin{align*}
\|\Pi_j\mathcal{E}_5^n\|_{L^2(\Om;\dot{H}^{\delta}(\O))}^2
&=
\|A^{\frac{\delta}{2}}\Pi_j\mathcal{E}_5^n\|_{L^2(\Om;L^2(\O))}^2\\
&=
\int_{t_1}^{t_n} \|A^{\frac{\delta}{2}}\Pi_j(e^{-(t_n-s)A}-e^{-(t_n-s)A}P_M)\|_{L^2(\Om;\mathcal{L}_2^0)}^2\d s\\
%&\le
%\int_0^{t_n} \sum_{k=2^{j-1}}^{2^j-1}\mu_k\|A^{\frac{\delta}{2}}(e^{-(t_n-s)A}-e^{-(t_n-s)A}P_M)\phi_k\|_{L^2(\Om;L^2(\O))}^2\d s\\
%& \mbox{\b (If $k\le M$, then $(I-P_M)\phi_k=0$, which does not affect the estimates below.) }
%\\
&\le
C\int_0^{t_n-M^{-2}}\sum_{k=2^{j-1}}^{2^j-1}\mu_kM^{-2\alpha}(t_n-s)^{-(\alpha+\delta+1-\alpha)}\|\phi_k\|_{\dot{H}^{-(1-\alpha)}}^2\d s\\
&\quad +
C\int_{t_n-M^{-2}}^{t_n}\sum_{k=2^{j-1}}^{2^j-1}\mu_k(t_n-s)^{-(\delta+1-\alpha)}\|\phi_k\|_{\dot{H}^{-(1-\alpha)}}^2\d s\\
&\le
CM^{-2\alpha}\int_0^{t_n-M^{-2}}(t_n-s)^{-(\delta+1)}\d s\sum_{k=2^{j-1}}^{2^j-1}\mu_k\lambda_k^{\alpha-1} \\
&\quad +
C\int_{t_n-M^{-2}}^{t_n}(t_n-s)^{-(\delta+1-\alpha)}\d s\sum_{k=2^{j-1}}^{2^j-1}\mu_k\lambda_k^{\alpha-1}\\
&\le
C\Big(M^{-2\alpha}(M^{2\delta}-t_n^{-\delta})  +
CM^{-2\alpha+2\delta}\Big)
\sum_{k=2^{j-1}}^{2^j-1}\mu_k\lambda_k^{\alpha-1},
\end{align*}
where 
$$
{\sum_{k=2^{j-1}}^{2^j-1}\mu_k\lambda_k^{\alpha-1}
\le
C\Big\| \int_0^Te^{-(T-s)A}\d W(s) \Big\|_{B^\infty L^2(\Omega;\dot H^\alpha(\mathcal{O}))} }\le C.
$$
As a consequence,
$$
\|\Pi_j\mathcal{E}_5^n\|_{L^2(\Om;\dot{H}^{\delta}(\O))}^2
\le
CM^{-2\alpha+2\delta}.
$$
On the other hand, by choosing a constant $\delta_0$ such that $\delta_0-\delta>0$, we have 
\begin{align*}
\|\Pi_j\mathcal{E}_5^n\|_{L^2(\Om;\dot{H}^{-\delta}(\O))}^2
&=
\int_0^{t_n} \sum_{k=2^{j-1}}^{2^j-1}\mu_k\|A^{-\frac{\delta}{2}}(e^{-(t_n-s)A}-e^{-(t_n-s)A}P_M)\phi_k\|_{L^2(\Om;L^2(\O))}^2\d s\\
&\le 
CM^{-2\alpha-2\delta_0}\int_0^{t_n-M^{-2}}(t_n-s)^{-(\alpha+\delta_0-\delta+1-\alpha)}\d s\sum_{k=2^{j-1}}^{2^j-1}\mu_k\lambda_k^{\alpha-1}\\
&\quad +
{CM^{-2\delta_0}\int_{t_n-M^{-2}}^{t_n}(t_n-s)^{-(\delta_0-\delta+1-\alpha)}\d s\sum_{k=2^{j-1}}^{2^j-1}\mu_k\lambda_k^{\alpha-1}}\\
&\le
C\Big(\! M^{-2\alpha-2\delta_0} (M^{2(\delta_0-\delta)} \!\! - \!  t_n^{-(\delta_0-\delta)})/(\delta_0-\delta)
\!+\!
M^{-2\alpha-2\delta}
\Big) \hspace{-5pt}
\sum_{k=2^{j-1}}^{2^j-1}\mu_k\lambda_k^{\alpha-1}\\
&\le
CM^{-2\alpha-2\delta} .
\end{align*}

Case 2: $t_n\le M^{-2}$. In this case, the integral from $M^{-2}$ to $t_n$ vanishes in the estimates above. 
The integral from $0$ to $M^{-2}$ can be estimated similarly. 

Overall, in both cases, the following estimate holds for all $\delta\in(0,\alpha)$ and and $t_n\in [t_1,T]$:  
\begin{align*}
\|\Pi_j\mathcal{E}_5^n\|_{B^{\infty}L^2(\Om;\dot{H}^{\pm\delta}(\O))}
\le
CM^{-\alpha\pm\delta} .
\end{align*}
By using the real interpolation method and the results in Lemma \ref{real-interpolation2}, we obtain  
\begin{align}\label{estimate-E5}
\|\mathcal{E}_5^n\|_{L^2(\Om;L^2(\O))}
=
\|\mathcal{E}_5^n\|_{(B^{\infty}L^2(\Om;\dot{H}^{-\delta}(\O)),B^{\infty}L^2(\Om;\dot{H}^{\delta}(\O)))_{\frac12,2}}
\le
CM^{-\alpha} .
\end{align}

\subsection{Completion of the proof}\label{section:completion}
To conclude, by substituting the estimates of $\mathcal{E}_k$, $k=1,2,\cdots,5$, into \eqref{error-repr}, 
we obtain that 
\begin{align}\label{final-error1}
&\|U_M^n-u(t_n)\|_{L^2(\Omega;L^2(\mathcal{O}))} \notag\\
&\le
CM^{-\alpha}t_n^{-\frac{\alpha-\beta}{2}}\|u^0\|_{ L^2(\Omega;\dot{H}^{\beta}(\mathcal{O}))}
+\sum_{j=2}^n\tau_j\|U_M^{j-1}-u(t_{j-1})\|_{L^2(\Om;L^2(\O))}
%&+
%C\tau \sum_{j=2}^n \|U_M^{j-1}-u(t_{j-1})\|_{L^2(\Omega;L^2(\mathcal{O}))}
%+C\sum_{j=2}^n\int_{t_{j-1}}^{t_j}M^{-\frac{\alpha}{d}}(t_n-s)^{-\frac{\alpha}{2}}\|f(u(t_{j-1}))\|_{L^2(\Om;L^2(\O))}\d s
+C\tau^{\alpha} \notag\\
&\le
CM^{-\alpha}t_n^{-\frac{\alpha-\beta}{2}}\|u^0\|_{ L^2(\Omega;\dot{H}^{\beta}(\mathcal{O}))}
+\sum_{j=1}^{n-1}\tau_j\|U_M^{j}-u(t_{j})\|_{L^2(\Om;L^2(\O))} +C\tau^{\alpha} 
\quad\mbox{for}\,\,\, n\ge 2 .
\end{align}

By comparing the first relation of \eqref{FD-method-2} with \eqref{mild-solution}, we obtain that 
\begin{align*}
&\hspace{-10pt}\tau_1\|U_M^{1}-u(t_{1})\|_{L^2(\Om;L^2(\O))} \notag\\ 
\le&\ 
\tau_1\| e^{-\tau_1A}U_M^0 - e^{-\tau_1A}u^0 \|_{L^2(\Om;L^2(\O))} \notag\\
&\ + \tau_1\Big\| \int_0^{\tau_1} e^{-(\tau_1-s)A}f(u(s))\d s+\int_0^{\tau_1} e^{-(\tau_1-s)A}\d W(s) \Big\|_{L^2(\Om;L^2(\O))} \notag\\
\le&\ 
C{(\tau_1^{1+\frac{\beta}{2}}+\tau_1)}\| u^0 \|_{ L^2(\Omega;\dot H^{\beta}(\O))}
+ \tau_1\Big\| \int_0^{\tau_1} (C+Cs^{\frac{\beta}{2}}) \d s\Big\|_{L^2(\Om;L^2(\O))} 
+C\tau_1^{1+\frac{\alpha}{2}} \notag\\
\le&\ 
C\tau_1^{1+\frac{\beta}{2}} +{C\tau_1}+ C\tau_1^{2+\frac{\beta}{2}} + C\tau_1^{1+\frac{\alpha}{2}} 
\le 
C\tau^{\frac{1}{1-\gamma}(1+\frac{\beta}{2})}{+C\tau^{\frac{1}{1-\gamma}}}  ,
\end{align*}
where we have used the property $\tau_1=O(\tau^{\frac{1}{1-\gamma}})$. 
Since condition \eqref{gamma-condition1} implies $\frac{1}{1-\gamma}(1+\frac{\beta}{2})\ge \alpha$ and $\frac{1}{1-\gamma}>2>\alpha$,
%{\color{red}(When $\beta>0$ this doesn't hold. Orelse we could change this into the bound of $CM^{-\alpha}$)} 
it follows that 
\begin{align*}
\tau_1\|U_M^{1}-u(t_{1})\|_{L^2(\Om;L^2(\O))} 
\le 
C\tau^{\alpha}  .
\end{align*}
Substituting this into \eqref{final-error1} yields
\begin{align}\label{final-error}
&\|U_M^n-u(t_n)\|_{L^2(\Omega;L^2(\mathcal{O}))} \notag\\
&\le
CM^{-\alpha}t_n^{-\frac{\alpha-\beta}{2}} 
+\sum_{j=2}^{n-1}\tau_j\|U_M^{j}-u(t_{j})\|_{L^2(\Om;L^2(\O))} +C\tau^{\alpha} 
\quad\mbox{for}\,\,\, n\ge 2 . 
\end{align}
Then, by applying the discrete Gronwall's inequality, we obtain  
\begin{align*}
\|U_M^n-u(t_n)\|
\le 
CM^{-\alpha}t_n^{-\frac{\alpha-\beta}{2}} +C\tau^{\alpha}
\quad\mbox{for}\,\,\, n\ge 2.
%It is remarkably, 
%\begin{cases}
%Ch^{\alpha}t_n^{-\frac{r+\beta}{2}}\quad \quad \mbox{if}\; r\neq \alpha,\\
%Ch^{\alpha}\sqrt{\ln\frac{t}{h^2}}t_n^{-\frac{\alpha-\beta}{2}} \quad \mbox{if}\; r=\alpha.
%\end{cases}
\end{align*}
This proves the result of Theorem \ref{main-theorem}. 
\endproof

%\end{proof}
%It suffices to prove
%\begin{align*}
% \|U_h^n-u^n\|_{L^2(\Omega;L^2(\mathcal{O}))}
%\le Ch^{\alpha}t_n^{-\frac{\alpha-\beta}{2}},
%\end{align*}
%since Theorem \ref{Temproal-discrete} and
%\begin{align*}
%\|U_h^n-u(t_n)\|_{L^2(\Omega;L^2(\mathcal{O}))}
%&\le
% \|U_h^n-u^n\|_{L^2(\Omega;L^2(\mathcal{O}))}
% +
% \|u^n-u(t_n)\|_{L^2(\Omega;L^2(\mathcal{O}))}.
%\end{align*}
%First, by the Lemma \ref{Thomee-upgrade} we have
%\begin{align*}
%\|U_h^1-u^1\|_{L^2(\Omega;L^2(\mathcal{O}))}
%&\le
%\|e^{-\tau_1 A_h}P_hu^0-e^{-\tau_1 A}u^0\|_{L^2(\Omega;L^2(\mathcal{O}))}\\
%&\le
%Ch^r\tau_1^{-\frac{\alpha-\beta}{2}}\|u^0\|_{\dot{H}^{-\beta}}.
%\end{align*}
%For $n\ge 2$ it follows that 
%\begin{align*}
%&\|U_h^n-u^n\|_{L^2(\Omega;L^2(\mathcal{O}))}\\
%&\le
%\|e^{-t_nA_h}P_hu^0-e^{-t_nA}u^0\|_{L^2(\Omega;L^2(\mathcal{O}))}\\
%&+
%\sum_{j=2}^n\int_{t_{j-1}}^{t_j}\big\|e^{-(t_n-s)A_h}P_hf(U_h^{j-1})-e^{-(t_n-s)A_h}P_hf(u^{j-1})\big\|_{L^2(\Omega;L^2(\mathcal{O}))}ds\\
%&+
%\sum_{j=2}^n\int_{t_{j-1}}^{t_j}\big\|e^{-(t_n-s)A_h}P_hf(u^{j-1})-e^{-(t_n-s)A}f(u^{j-1})\big\|_{L^2(\Omega;L^2(\mathcal{O}))}ds\\
%&+
%\end{align*}
%\end{proof}
%

\section{Numerical experiments}\label{numerical}
In this section, we present numerical results to support the theoretical analysis. 
All computations are performed by Matlab with double precision (see \cite{Higham-2021} for algorithmic implementation on Matlab for stochastic differential equations). 

Let $\O=[0,1]$ and $T=0.5$. 
We solve problem \eqref{SPDE1} by the proposed modified exponential Euler scheme with Fourier collocation method in \eqref{FD-method-2}, with the nonlinear drift function  
$$
f(u)=\sqrt{1+u^2},
$$
which satisfies items (1)--(2) in Assumption \ref{thm-assumption}.
The following two deterministic initial values are tested
\begin{align*}
u_1^0(x) = \sin(\pi x)
\quad\mbox{and} \quad 
u_2^0(x) = \delta\Big(x-\frac12\Big),
%\quad 
%u_3^0(x) = \sum_{k=1}^{\infty} k^{1-0.01} \sin(k\pi x),
\end{align*} 
where $u_1^0\in H_0^1(\O)\cap C^{\infty}(\O)$ and $u_2^0\in\dot{H}^{-\frac12-\eps}(\O)$ is the Dirac delta function, where $\eps>0$ can be an arbitrary small number.
% $u_1^0\in H_0^1(\O)\cap C^{\infty}(\O)$ and $u_3^0\in \dot{H}^{-\frac32}(\O)$. 

The implementation of the numerical method is simple, i.e., the nonlinear term can be calculated by 
$$
\int_{t_{n-1}}^{t_n}e^{-(t_n-s)A}I_Mf(U_M^{n-1})\d s
=
\left(\frac{1-e^{-\tau_nA}}{A}\right)I_Mf(U_M^{n-1}),
$$
where $I_M$ can be calculated by using FFT with $O(M\ln M)$ operations at every time level. 

The noise term can be calculated by
\begin{align*}
\int_{t_{n-1}}^{t_n} e^{-(t_n-s)A}P_M\d W(s)
%&=
%\int_{t_{n-1}}^{t_n}e^{-(t_n-s)A}\sum_{k=1}^M\sqrt{\mu_k}\phi_k \d W_k(s)\\
&=
\sum_{k=1}^M \sqrt{\mu_k}\phi_k\int_{t_{n-1}}^{t_n}e^{-(t_n-s)\lambda_k}\d W_k(s)\\
&=
\sum_{k=1}^M\sqrt{\mu_k} \phi_k\left(\frac{1-e^{-2\tau_n\lambda_k}}{2\lambda_k}\right)^{\frac12}\xi_k^n 
\end{align*}
with independent and standard normally distributed random variables $\xi_k^n$ for $1\le k\le M$ and $1\le n\le N$.
%Except for the high efficient employment of fast sine transform in scheme \eqref{FD-method-2},  
%it is remarkable that
%this scheme could also be implemented quite simply by referring to \cite[section 3]{Jentzen-Kloeden-2009}, where the relation 
%$$
%\int_{t_{n-1}}^{t_n} e^{-(t_n-s)\lambda_k}dW_k(s)=\left(\frac{1-e^{-2\tau_n\lambda_k}}{2\lambda_k}\right)^{\frac12}\xi_k 
%$$
% is given with independent and standard normally distributed random variables $\xi_k$ for all $k\in \mathbb{N}^+$.
If the noises $\xi_k^n $ are generated with a fine mesh in time, then the following identity can be used to calculate the numerical solution with a coarse mesh in time (with stepsize $t_{n+m}-t_n$): 
\begin{align*}
\int_{t_n}^{t_{n+m}}e^{-(t_{n+m}-s)\lambda_k}\d W_k(s)
%&=
%\sum_{j=1}^m \int_{t_{n+j-1}}^{t_{n+j}} e^{-(t_{n+m}-s)\lambda_k}\d W_k(s)\\
&=
\sum_{j=1}^me^{-(t_{n+m}-t_{n+j})\lambda_k}\int_{t_{n+j-1}}^{t_{n+j}}e^{-(t_{n+j}-s)\lambda_k}\d W_k(s)\\
&=
\sum_{j=1}^me^{-(t_{n+m}-t_{n+j})\lambda_k} \left(\frac{1-e^{-2\tau_{n+j}\lambda_k}}{2\lambda_k}\right)^{\frac12}\xi_k^{n+j}
\quad\forall\, m\ge 1. 
\end{align*}
This allows us to test the errors and convergence orders by using a reference solution with a very fine mesh in time. 

%{\color{red} We consider different initial data in Example 1 and Example 2}
%
%下面还要写怎么实现非均匀网格下的同步伐的
%然后就是一些分析了
%{
%1. 解释不同的表格什么意思
%2. u 固定一个变量去看另一个变量
%3. 为什么会选取这些不同的 mu k 
%}

%\noindent{\bf Example 1}
To test the spatial convergence orders, we fix a sufficiently small time stepsize $\tau=2^{-10}$ and  calculate the error by 
$$
E_1(M) = \left(\frac{1}{I}\sum_{i=1}^I \|U_{\tau,M}^N(\omega_i)-U_{\tau,2M}^N(\omega_i)\|_{L^2(\O)}^2\right)^{\frac12}
$$
for $M = 16, 32, 64,128$, 
i.e., the expectations of errors over $I=1000$ samples at $t=T$,
and then present them in Tables \ref{space-table-u1}--\ref{space-table-u2}
 for different initial data. 
 
\begin{table}[htbp]%[H]
\caption{Spatial discretization error $E_1(M)$ with initial data $u^0_1$ and $\gamma = 0.7$}
\vspace{.03in}
\label{space-table-u1}
\centering
\setlength{\tabcolsep}{4mm}{
\begin{tabular}{cccccc}
      \toprule[1pt]
	$\mu_k{\Huge\backslash} M$ & 16 & 32& 64& 128&Order\\ %[2pt] 
        \midrule[0.7pt]
         $\mu_k \equiv 1$& %\frac{1}{k^2}
%3.905e-5&
%9.829e-6&
%2.461e-6&
%6.158e-7&
3.866e-2&
2.786e-2&
1.980e-2&
1.402e-2&
        $\approx    0.50\,(0.50)$\\ %[2.3pt] 
                        	 \midrule
        $\mu_k = 1/k^{0.5}$& %\frac{1}{k^2}
1.787e-2&
1.076e-2&
6.469e-3&
3.859e-3&
        $\approx   0.75\,(0.75)$\\ %[2.3pt] 
        	 \midrule
        $\mu_k = 1/k^{0.8}$& %\frac{1}{k^2}
%3.905e-5&
%9.829e-6&
%2.461e-6&
%6.158e-7&
1.126e-2&
6.188e-3&
3.326e-3&
1.780e-3&
        $\approx   0.90\,(0.90)$\\ %[2.3pt] 
        %\rule{0pt}{16pt}
	 \midrule
        $\mu_k = 1/k$& %\frac{1}{k^2}
8.243e-3&
4.238e-3&
2.126e-3&
1.071e-3&
        $\approx   0.99\,(1.00)$\\ %[2.3pt] 
        	 \midrule
        $\mu_k = 1/k^{1.1}$& %\frac{1}{k^2}
7.156e-3&
3.506e-3&
1.710e-3&
8.311e-4&
        $\approx   1.04\,(1.00)$\\ %[2.3pt] 
        %\rule{0pt}{16pt}
%	 \midrule
%	$\mu_k= 1/k^2$ & 
%1.813e-3&
%6.547e-4&
%2.339e-4&
%8.342e-5&
%	 $\approx 1.50\,(1.50)$ \\ %[2.3pt]
%	  \midrule
%	$\mu_k= 1/k^3$ & 
%4.575e-4&
%1.174e-4&
%3.006e-5&
%7.559e-6&
%	 $\approx 1.99\,(2.00)$ \\ %[2.3pt]  
%	 \midrule
%	 	$\mu_k= 1/k^{20}$ & 
%2.120e-4&
%5.695e-5&
%1.477e-5&
%3.761e-6&
%	 $\approx 1.97$ \\ %[2.3pt] 
         \bottomrule
\end{tabular}}
\end{table}

\begin{table}[htbp]%[H]
\caption{Spatial discretization error $E_1(M)$ with initial data $u^0_2$ and $\gamma = 0.7$}
\vspace{.03in}
\label{space-table-u2}
\centering
\setlength{\tabcolsep}{4mm}{
\begin{tabular}{cccccc}
      \toprule[1pt]
	$\mu_k{\Huge\backslash} M$ & 16 & 32& 64& 128&Order\\ %[2pt] 
        \midrule[0.7pt]
         $\mu_k \equiv 1$& %\frac{1}{k^2}
%3.905e-5&
%9.829e-6&
%2.461e-6&
%6.158e-7&
3.858e-2&
2.776e-2&
1.971e-2&
1.403e-2&
        $\approx    0.49\,(0.50)$\\ %[2.3pt] 
                        	 \midrule
        $\mu_k = 1/k^{0.5}$& %\frac{1}{k^2}
1.787e-2&
1.080e-2&
6.443e-3&
3.853e-3&
        $\approx   0.74\,(0.75)$\\ %[2.3pt] 
        	 \midrule
        $\mu_k = 1/k^{0.8}$& %\frac{1}{k^2}
1.127e-2&
6.164e-3&
3.310e-3&
1.788e-3&
        $\approx   0.89\,(0.90)$\\ %[2.3pt] 
        %\rule{0pt}{16pt}
	 \midrule
        $\mu_k = 1/k$& %\frac{1}{k^2}
8.261e-3&
4.205e-3&
2.139e-3&
1.072e-3&
        $\approx   1.00\,(1.00)$\\ %[2.3pt] 
        	 \midrule
        $\mu_k = 1/k^{1.1}$& %\frac{1}{k^2}
7.130e-3&
3.511e-3&
1.705e-3&
8.293e-4&
        $\approx   1.04\,(1.00)$\\ %[2.3pt] 
        %\rule{0pt}{16pt}
%	 \midrule
%	$\mu_k= 1/k^2$ & 
%1.827e-3&
%6.615e-4&
%2.341e-4&
%8.364e-5&
%	 $\approx 1.49\,(1.50)$ \\ %[2.3pt]
%	  \midrule
%	$\mu_k= 1/k^3$ & 
%4.494e-4&
%1.176e-4&
%2.989e-5&
%7.558e-6&
%	 $\approx 1.98\,(2.00)$ \\ %[2.3pt]  
%	 \midrule
%	 	$\mu_k= 1/k^{20}$ & 
%2.104e-4&
%5.640e-5&
%1.459e-5&
%3.704e-6&
%	 $\approx 1.98$ \\ %[2.3pt] 
         \bottomrule
\end{tabular}}
\end{table}

\bigskip

In the case $\mu_k=1/k^\delta$ {($0\le\delta<1$)}, the noise satisfies Assumption \ref{thm-assumption} (3) 
 with $\alpha=\frac{1+\delta}{2}$. This order of convergence in space is well illustrated by the numerical results in Tables \ref{space-table-u1}--\ref{space-table-u2}. 
%This not only contains the important case of space-time white noise $(\delta=0)$, but also 
%provides a type of noise that is not of trace class however still brings first-order convergence in space $(\delta=1)$, for which we only obtain the convergence rate of $1-$.
%While driven by trace-class noise of the case $\mu_k=\frac{1}{k^{1+\eps}}$ for a small $\eps>0$ the scheme \eqref{FD-method-2}  is also shown to converge with the rate of 1 in space.
From Tables \ref{space-table-u2} we see that the modified exponential Euler method with Fourier collocation method in space is robust with respect to the regularity of the initial data, including measure-valued functions such as the Diract delta function.
% does not affect the spatial discretization errors much (note that the regularity of function $u_3^0$ is too bad to satisfy the condition of Theorem \ref{main-theorem}).

%Note that $\delta = 0$
%And for the case of $\mu_k = \frac{1}{k^{1+\eps}}$

%Combing the inequality \eqref{little-delight}, 
%it is clear that the convergence orders shown in tables \ref{space-table-u1}--\ref{space-table-u3}
%for the cases of $\mu_k = \frac{1}{k^{\delta}}$ with $0\le\delta\le 1$
% are consistent with our theoretical analysis, i.e., 

To test the temporal convergence orders, we choose $M=N$ and  calculate the error by 
$$
E_2(\tau) = \left(\frac{1}{I}\sum_{i=1}^I \|U_{\tau,N}^N(\omega_i)-U_{\tau/2, N_2}^{N_2}(\omega_i)\|_{L^2(\O)}^2\right)^{\frac12},
$$
where $N_2$ is the number of the time levels for time stepsize $\tau/2$.
By Theorem \ref{main-theorem}, the spatial convergence order equals the temporal convergence order. %The
The numerical results are given in Tables \ref{time-table-u1}--\ref{time-table-u2}, 
where the observed temporal convergence orders are consistent with the theoretical result proved in Theorem \ref{main-theorem}. 
%Further, to avoid the domination of spatial errors, in Tables \ref{time-table-u1}--\ref{time-table-u3} we also compute errors of
%\be\label{E-2-tau}
%E_2(\tau) = \left(\frac{1}{I}\sum_{i=1}^I \|U_{\tau,M}^N(\omega_i)-U_{\tau/2, M}^N(\omega_i)\|_{L^2(\O)}^2\right)^{\frac12}
%\en
%for $\tau=\frac{1}{16}, \frac{1}{32}, \frac{1}{64}, \frac{1}{128}$ with $M=2^{10}$ fixed. However, the corresponding convergence orders are much higher than those calculated with the help of $E_2(\tau)$. Observing from Tables \ref{time-table-u1}--\ref{One-time-table-u3}, with $E_2(\tau)$ the orders go down when the noise is worse than the case of $\mu_k=\frac{1}{k}$ for $k\in \mathbb{N}^+$, whereas with $E_2(\tau)$ the orders decrease only when the noise is worse than the space-time white noise.
\begin{table}[htbp]
\caption{Temporal discretization error $E_2(\tau)$ with initial data $u_1^0$ and $\gamma = 0.7$ }
\vspace{.03in}
\label{time-table-u1}
\centering
\setlength{\tabcolsep}{4mm}{
\begin{tabular}{cccccc}
      \toprule[1pt]
	$\mu_k{\Huge\backslash} \tau$ &  1/16 &1/32&1/64&1/128&Order\\ %[2pt] 
        \midrule[0.7pt]
         $\mu_k \equiv1$& %\frac{1}{k^2}
2.401e-2&
1.665e-2&
1.192e-2&
8.476e-3&
        $\approx 0.49\,(0.50)$\\ %[2.3pt] 
        %\rule{0pt}{16pt}
                       	 \midrule
        $\mu_k = 1/k^{0.5}$& %\frac{1}{k^2}
8.661e-3&
5.055e-3&
3.040e-3&
1.824e-3&
        $\approx   0.74\,(0.75)$\\ %[2.3pt] 
        	 \midrule
        $\mu_k = 1/k^{0.8}$& %\frac{1}{k^2}
4.739e-3&
2.509e-3&
1.362e-3&
7.315e-4&
        $\approx   0.90\,(0.90)$\\ %[2.3pt] 
        %\rule{0pt}{16pt}
	 \midrule
        $\mu_k = 1/k$& %\frac{1}{k^2}
3.187e-3&
1.594e-3&
8.085e-4&
4.047e-4&
        $\approx   1.00\,(1.00)$\\ %[2.3pt] 
        	 \midrule
        $\mu_k = 1/k^{1.1}$& %\frac{1}{k^2}
2.671e-3&
1.284e-3&
6.283e-4&
3.051e-4&
        $\approx   1.04\,(1.00)$\\ %[2.3pt] 
        %\rule{0pt}{16pt}
%	 \midrule
%	$\mu_k= 1/k^2$ & 
%8.727e-4&
%4.679e-4&
%2.288e-4&
%1.060e-4&
%	 $\approx 1.11\,(1.00)$ \\ %[2.3pt]
         \bottomrule
\end{tabular}}
\end{table}

\begin{table}[htbp]
\caption{Temporal discretization error $E_2(\tau)$ with initial data $u_2^0$ and $\gamma = 0.7$ }
\vspace{.03in}
\label{time-table-u2}
\centering
\setlength{\tabcolsep}{4mm}{
\begin{tabular}{cccccc}
      \toprule[1pt]
	$\mu_k{\Huge\backslash} \tau$ & 1/16 &1/32&1/64& 1/128&Order\\ %[2pt] 
        \midrule[0.7pt]
         $\mu_k \equiv1$& %\frac{1}{k^2}
2.399e-2&
1.668e-2&
1.187e-2&
8.478e-3&
        $\approx 0.49\,(0.50)$\\ %[2.3pt] 
        %\rule{0pt}{16pt}
                       	 \midrule
        $\mu_k = 1/k^{0.5}$& %\frac{1}{k^2}
8.659e-3&
5.056e-3&
3.054e-3&
1.824e-3&
        $\approx   0.74\,(0.75)$\\ %[2.3pt] 
        	 \midrule
        $\mu_k = 1/k^{0.8}$& %\frac{1}{k^2}
4.728e-3&
2.501e-3&
1.361e-3&
7.319e-4&
        $\approx   0.90\,(0.90)$\\ %[2.3pt] 
        %\rule{0pt}{16pt}
	 \midrule
        $\mu_k = 1/k$& %\frac{1}{k^2}
3.202e-3&
1.587e-3&
8.072e-4&
4.042e-4&
        $\approx   1.00\,(1.00)$\\ %[2.3pt] 
        	 \midrule
        $\mu_k = 1/k^{1.1}$& %\frac{1}{k^2}
2.664e-3&
1.294e-3&
6.379e-4&
3.085e-4&
        $\approx   1.05\,(1.00)$\\ %[2.3pt] 
        %\rule{0pt}{16pt}
%	 \midrule
%	$\mu_k= 1/k^2$ & 
%8.709e-3&
%4.462e-4&
%2.161e-4&
%9.712e-5&
%	 $\approx 1.15\,(1.00)$ \\ %[2.3pt]
         \bottomrule
\end{tabular}}
\end{table}

\section{Conclusions}\label{conclusions}

We have considered a modified exponential Euler method for the semilinear stochastic heat equation, with Fourier Galerkin and Fourier collocation method in space. 
Some new techniques are introduced to the error analysis, including the stochastic Besov spaces and its interpolation properties to characterize the noises, and a class of locally refined variable stepsizes to resolve the singularity of the solution at $t=0$.
By using these new techniques, we have proved that the method has $\alpha$th-order convergence for initial data in $L^4(\Omega;H^\beta(\O))$ with $\beta\in(-1,\alpha]$, for a class of noises characterized by a parameter $\alpha\in(0,1]$, which includes trace-class noises (with $\alpha=1$) and one-dimensional space-time white noises (with $\alpha=\frac12$). 
The numerical results also support the theoretical analysis. 

{ In the numerical schemes of \eqref{FD-method-1} and \eqref{FD-method-2}, we have used variable stepsizes and modified the exponential integrator at the initial time level to address the singularity of the solution at $t=0$. This is needed when $\beta<0$ because the initial data $u^0$ may not be a pointwisely defined function and therefore the term $f(u^0)$ in the standard exponential integrator may not be pointwisely well-defined. 
%strategy employed in this manuscript can not be used (see the decomposition of $U_M^n-u(t_n)$ at the beginning of Section 3.2). Therefore the modification of the initial step is very necessary in this case.
In the case $0\le\beta\le \alpha$, the variable stepsize and the modification of the initial step may not be necessary. However, since the estimate of $\mathcal{E}_{4,1}^n$ in \eqref{estimate-E41} involves $t_{j-1}^{-(\alpha-\beta)/2}$, the estimation of this term at the initial step needs to be changed to a different way in the case $0\le\beta<\alpha$. 

In estimates \eqref{E2n-estimate-1}--\eqref{E2n-estimate-2}, we have used the additional assumption that in the case $\alpha=1$ the noise is trace class; see in Assumption \ref{thm-assumption} (3). This is only needed for the Fourier sine collocation method in \eqref{FD-method-2} with trigonometric interpolation. 
Theorem \ref{main-theorem-0} (for the spectral Galerkin method) still holds without requiring the noise to be trace class in the case $\alpha=1$. This is because that $e^{-(t_n-s)A}$ commutes with $P_M$ and therefore \eqref{E2n-estimate-1} can be estimated in the following different way: 
\begin{align}
\|\mathcal{E}_2^n\|_{L^2(\Om;L^2(\O))} 
&=\sum_{j=2}^n\int_{t_{j-1}}^{t_j}\big(P_M e^{-(t_n-s)A}f(U_M^{j-1})-e^{-(t_n-s)A}f(U_M^{j-1})\big)ds \notag\\
&\le%1
C\sum_{j=2}^n\int_{t_{j-1}}^{t_j}
M^{-1} \|e^{-(t_n-s)A}f(U_M^{j-1})\|_{L^2(\Om;\dot H^{1}(\O))} \d s \notag\\
&\le%2
C\sum_{j=2}^n\int_{t_{j-1}}^{t_j}M^{-1} (t_n-s)^{-\frac{1}{2}} \|f(U_M^{j-1})\|_{L^2(\Om;L^2(\O))}  \d s
 \notag\\
&\le%2
C\sum_{j=2}^n\int_{t_{j-1}}^{t_j}M^{-1} (t_n-s)^{-\frac{1}{2}} (\|f(0)\|_{L^2(\Om;L^2(\O))} 
+\|U_M^{j-1}\|_{L^2(\Om;L^2(\O))} ) \d s\notag
% \notag\\
%&\le%2
%C\sum_{j=2}^n\int_{t_{j-1}}^{t_j}M^{-1} (t_n-s)^{-\frac{1}{2}} (1+t_{j-1}^{\frac{\beta}{2}} ) \d s  \notag\\
%&\le%2
%C\sum_{j=2}^n\int_{t_{j-1}}^{t_j}M^{-1} (t_n-s)^{-\frac{1}{2}} (1+s^{\frac{\beta}{2}} ) \d s  \notag\\
%&\le%3
%CM^{-1}  ,
\end{align}
%where the second to last inequality uses the smoothing property of the analytic semigroup $e^{-tA}$, i.e., 
where we have used the smoothing property of the analytic semigroup $e^{-tA}$, i.e., 
$$
\|e^{-(t_n-s)A}g\|_{\dot H^{1}(\O)}\le C(t_n-s)^{-\frac12}\|g\|_{L^2(\O)} \quad\mbox{for}\,\,\, g\in L^2(\O) . 
$$
{From \eqref{fullyUM1} we see that   
\begin{align}\label{estimate-UM-1}
\|U_M^1\|_{L^2(\Om;L^2(\O))} 
=
\|e^{-\tau_1A}P_Mu^0\|_{L^2(\Om;L^2(\O))} 
&=
\|A^{-\frac{\beta}{2}}e^{-\tau_1A}P_Mu^0\|_{L^2(\Omega;\dot{H}^{\beta}(\O))}\notag\\
&\le
C(1+t_1^{\frac{\beta}{2}})\|u^0\|_{L^2(\Omega;\dot{H}^{\beta}(\O))} . 
\end{align}
From \eqref{fullyUMn} we can obtain the following expression of $U_M^n$ similarly as \eqref{expr-UMn} (with $I_M$ replaced by $P_M$ therein):
\begin{align*}%\label{expr-UMnPM}
U_M^n 
&=
 e^{-(t_n-\tau_1) A}U_M^{1}
 +
\sum_{j=2}^n \int_{t_{j-1}}^{t_j}e^{-(t_n-s)A}P_Mf(U_M^{j-1})ds
 +
 \int_{t_1}^{t_n}e^{-(t_n-s)A}P_MdW(s) ,
\end{align*}
which implies that 
\begin{align*}
&\|U_M^n\|_{L^2(\Om;L^2(\O))} \\
&\le  
 \|e^{-(t_n-\tau_1) A}U_M^{1}\|_{L^2(\Omega;L^2(\O))}
%\|A^{-\frac{\beta}{2}}e^{-t_n A}P_Mu^0\|_{L^2(\Omega;\dot{H}^{\beta}(\O))}
+
\sum_{j=2}^n\tau_j \|f(U_M^{j-1})\|_{L^2(\Om;L^2(\O))} 
+t_n^{\frac{\alpha}{2}} 
\quad\mbox{(here \eqref{t-alpha} is used)} \\
&\le 
%(1+t_n^{\frac{\beta}{2}})\|u^0\|_{L^2(\Omega;\dot{H}^{\beta}(\O))}
\|A^{-\frac{\beta}{2}}e^{-t_n A}P_Mu^0\|_{L^2(\Omega;\dot{H}^{\beta}(\O))}
+\sum_{j=2}^n\tau_j\left(\|U_M^{j-1}\|_{L^2(\Om;L^2(\O))} +\|f(0)\|_{L^2(\Om;L^2(\O))}\right)
+t_n^{\frac{\alpha}{2}}\\
&\le
%C(1+t_n^{\frac{\beta}{2}})
(1+t_n^{\frac{\beta}{2}})\|u^0\|_{L^2(\Omega;\dot{H}^{\beta}(\O))}
+\sum_{j=2}^n\tau_j\|U_M^{j-1}\|_{L^2(\Om;L^2(\O))} \quad\quad\mbox{for}\,\,\, 2\le n\le N , 
\end{align*}
where we have used %\eqref{estimate-UM-1} and
 the Lipschitz continuity of $f$ in the second to last inequality. 
Applying the discrete Gronwall inequality to this equation, together with equation \eqref{estimate-UM-1}, we can derive
\begin{align*}
\|U_M^n\|_{L^2(\Om;L^2(\O))} 
\le 
C(1+t_n^{\frac{\beta}{2}}) 
\quad\mbox{for}\,\,\, 1\le n\le N. 
\end{align*}
Therefore,}
\begin{align}\label{E2n-estimate-SpG} 
\|\mathcal{E}_2^n\|_{L^2(\Om;L^2(\O))} 
&\le%2
C\sum_{j=2}^n\int_{t_{j-1}}^{t_j}M^{-1} (t_n-s)^{-\frac{1}{2}} (1+t_{j-1}^{\frac{\beta}{2}} ) \d s  \notag\\
&\le%2
C\sum_{j=2}^n\int_{t_{j-1}}^{t_j}M^{-1} (t_n-s)^{-\frac{1}{2}} (1+s^{\frac{\beta}{2}} ) \d s  \notag\\
&\le%3
CM^{-1}.
\end{align}

%Hence, Theorem \ref{main-theorem-0} holds without requiring the noise to be trace class in the case $\alpha=1$. 

For the Fourier sine collocation method in \eqref{FD-method-2}, if the noise is not trace class in the case $\alpha=1$, Theorem \ref{main-theorem} can still proved by using the inverse inequality (proof is omitted)
$$
\|P_Mu(t_{j-1})\|_{L^2(\Om;\dot H^{1}(\O))}
\le C(\ln M)^{\frac12} \|P_Mu(t_{j-1})\|_{B^\infty L^2(\Om;\dot H^{1}(\O))} . 
$$
This loses a logarithmic order of convergence in the case $\alpha=1$. 
}

\section*{Appendix: Proof of Proposition \ref{Regularity}}
\renewcommand{\theequation}{A.\arabic{equation}}
\renewcommand{\thelemma}{A.\arabic{lemma}}

\subsection*{A.1\, Existence and uniqueness}
We prove the existence and uniqueness of mild solutions by using the Banach fixed point theorem.

For $v\in X=\big\{v\in L^1\big(0,T; L^2(\Om;L^2(\O))\big): \sup\limits_{t\in (0,T]} (1+t^{\frac{\beta}{2}})^{-1}\|v(t)\|_{L^2(\Om;L^2(\O))}<\infty \big \}$
we define a nonlinear operator $M:X\to X$ by 
$$
Mv(t)
=
e^{-tA}u^0+\int_0^te^{-(t-s)A}f(v(s))\d s+\int_0^t e^{-(t-s)A}\d W(s) ,
$$
which is well-defined as  
\begin{align*}
\|Mv(t)\|_{L^2(\Om;L^2(\O))}
&\le 
C(1+t^{\frac{\beta}{2}}) \|u^0\|_{ L^2(\Omega;\dot{H}^{\beta}(\O))}
+
C\int_0^t (1+\|v(s)\|_{L^2(\Om;L^2(\O))}) \d s
+
Ct^{\frac{\alpha}{2}}\\
&\le
C(1+t^{\frac{\beta}{2}}) \|u^0\|_{ L^2(\Omega;\dot{H}^{\beta}(\O))}
+
C\int_0^t(1+s^{\frac{\beta}{2}}) \|v\|_{X} \d s
+
Ct^{\frac{\alpha}{2}} . 
\end{align*}
 %we have $Mv\in X$. Therefore $M:X\to X$ is a well-defined operator.
We consider the space $X_{\lambda}$, which is defined as the vector space $X$ with the equivalent norm 
$$
\|v\|_{X_{\lambda}} : =\sup_{t\in (0,T]}e^{-\lambda t}(1+t^{\frac{\beta}{2}})^{-1}\|v(t)\|_{L^2(\Om;L^2(\O))}, 
$$
 where $\lambda\ge 1$ is a fixed constant to be determined later. 
Therefore, $v\in X_{\lambda}$ if and only if $v\in X$. 
If $v_1,v_2\in X_{\lambda}$ then 
\begin{align*}
&e^{-\lambda t}(1+t^{\frac{\beta}{2}})^{-1}\|Mv_1(t)-Mv_2(t)\|_{L^2(\Om;L^2(\O))}\\
&\le 
{C}\int_0^t e^{-\lambda t} \| f(v_1(s)) - f(v_2(s)) \|_{L^2(\Om;L^2(\O))}\d s \\
 &\le
C\int_0^t e^{-\lambda (t-s)}(1+s^{\frac{\beta}{2}})  e^{-\lambda s} (1+s^{\frac{\beta}{2}})^{-1} \|v_1(s)-v_2(s)\|_{L^2(\Om;L^2(\O))}\d s\\
&\le 
C\|v_1-v_2\|_{X_{\lambda}} \int_0^t e^{-\lambda (t-s)}(1+s^{\frac{\beta}{2}}) \d s\\
&\le
C\|v_1-v_2\|_{X_{\lambda}}\int_{0}^{\lambda t}e^{-\lambda t +\delta}
(\lambda^{-1}+\lambda^{-\frac{\beta}{2}-1}\delta^{\frac{\beta}{2}})\d \delta \\
&\le 
C\|v_1-v_2\|_{X_{\lambda}}
\Big(\frac{1-e^{-\lambda t}}{\lambda}+
\lambda^{-\frac{\beta}{2}-1} \int_{0}^{1}e^{-\lambda t+\delta}\delta^{\frac{\beta}{2}}\d \delta 
+
\lambda^{-\frac{\beta}{2}-1} \int_{1}^{\max\{\lambda t,1\}}e^{-\lambda t+\delta}\delta^{\frac{\beta}{2}}\d \delta 
\Big)\\
&{\le 
C\|v_1-v_2\|_{X_{\lambda}}
\Big(
\lambda^{-1} 
+\lambda^{-\frac{\beta}{2}-1} \int_{0}^{1}\delta^{\frac{\beta}{2}}\d \delta 
+\lambda^{-\frac{\beta}{2}-1} \int_{1}^{\max\{\lambda t,1\}}e^{-\lambda t+\delta} \delta^{\frac{\beta}{2}}  \d \delta 
\Big).}%\\
%&\le
%C\lambda^{\frac{\beta}{2}-1}\|v_1-v_2\|_{X_{\lambda}},
\end{align*}
Since
\begin{align*}
\lambda^{-\frac{\beta}{2}-1} 
\mbox{$\int_{1}^{\max\{\lambda t,1\}}$} e^{-\lambda t+\delta} \delta^{\frac{\beta}{2}}  \d \delta
&\le
\left\{\begin{aligned}
&\lambda^{-\frac{\beta}{2}-1} \mbox{$\int_{1}^{\max\{\lambda t,1\}}$} e^{-\lambda t+\delta}  \d \delta
&&\mbox{if}\,\,\, \beta\in(-1,0] \\
&\lambda^{-\frac{\beta}{2}-1} \mbox{$\int_{1}^{\max\{\lambda t,1\}}$} e^{-\lambda t+\delta} (\lambda t)^{\frac{\beta}{2}} \d \delta
&&\mbox{if}\,\,\, \beta\in(0,\alpha] 
\end{aligned}\right. \\
&\le
\left\{\begin{aligned}
&\lambda^{-\frac{\beta}{2}-1} 
&&\mbox{if}\,\,\, \beta\in(-1,0]  \\
&\lambda^{-1} 
&&\mbox{if}\,\,\, \beta\in(0,\alpha] ,
\end{aligned}\right.
\end{align*}
it follows that
\begin{align*}
e^{-\lambda t}(1+t^{\frac{\beta}{2}})^{-1}\|Mv_1(t)-Mv_2(t)\|_{L^2(\Om;L^2(\O))} 
&\le
C\lambda^{-\frac{\min(\beta,0)}{2}-1}\|v_1-v_2\|_{X_{\lambda}} \\
&\le
C\lambda^{-\frac{1}{2}}\|v_1-v_2\|_{X_{\lambda}} \quad\mbox{for}\,\,\, \beta\in(-1,\alpha] . 
\end{align*}
Therefore, $M$ is a contraction map on $X_{\lambda}$ when $\lambda$ is sufficiently large. This and the Banach fixed point theorem imply that there exists a unique fixed point of $M$ on $X_{\lambda}=X$. This fixed point of $M$ is denoted by $u$, which is the mild solution of problem \eqref{SPDE1}. 
 
\subsection*{A.2\, Regularity}
%{\b 
%%We first prove the quantitive regularity results \eqref{u-L2-estimate} and \eqref{B-Inf-estimate} since the  
%We first prove the quantitative regularity results \eqref{u-L2-estimate} and \eqref{B-Inf-estimate}, since the former will be used in the proof of the qualitative regularity.}
%Since the proof of the qualitative regularity result relies on the quantitative result, we first prove \eqref{u-L2-estimate} and \eqref{B-Inf-estimate} first. 
By using the expression of the mild solution in \eqref{mild-solution} and the property of the noise in \eqref{t-alpha}, we have 
\begin{align*}
\|u(t)\|_{L^p(\Om;L^2(\O))} 
\le&\
\|A^{-\frac{\beta}{2}}e^{-tA}A^{\frac{\beta}{2}}u^0\|_{L^p(\Om;L^2(\O))}+
\big\| \mbox{$\int_0^t$} e^{-(t-s)A}f(u(s))ds\big\|_{L^p(\Om;L^2(\O))} \\
&\
+
\big\|\mbox{$\int_0^t$} e^{-(t-s)A}dW(s)\big\|_{L^p(\Om;L^2(\O))}
\end{align*}
Since $e^{-tA}$ is an analytic semigroup, it follows that 
\begin{align*}
\|A^{-\frac{\beta}{2}}e^{-tA}A^{\frac{\beta}{2}}u^0\|_{L^p(\Om;L^2(\O))}
\le
\left\{
\begin{aligned}
&Ct^{\frac{\beta}{2}} \|A^{\frac{\beta}{2}}u^0\|_{L^p(\Om;L^2(\O))} &&\mbox{if}\,\,\, \beta\in(-1,0],\\
&C \|A^{\frac{\beta}{2}}u^0\|_{L^p(\Om;L^2(\O))} &&\mbox{if}\,\,\, \beta\in(0,\alpha] . 
\end{aligned}
\right.
\end{align*}
{ And with the help of \eqref{ito-isometry} and \eqref{BDG} it follows that 
\begin{align*}
\big\| \mbox{$\int_0^t$} e^{-(t-s)A}dW(s)\big\|_{L^p(\Om;L^2(\O))}
\le 
C\big\| \mbox{$\int_0^t$} e^{-(t-s)A}dW(s)\big\|_{L^2(\Om;L^2(\O))}
\end{align*}
holds for $p\ge 2$.
}
Therefore, 
\begin{align*}
\|u(t)\|_{L^p(\Om;L^2(\O))} 
\le
&\ 
C(1+t^{\frac{\beta}{2}})\|u^0\|_{ L^p(\Omega;\dot{H}^{\beta}(\O))}
+
\mbox{$\int_0^t$} \|f(u(s))\|_{L^p(\Om;L^2(\O))}\d s+Ct^{\frac{\alpha}{2}} \\
\le&\ 
C(1+t^{\frac{\beta}{2}})
+
\mbox{$\int_0^t$} ({\|f(0)\|_{L^p(\Om;L^2(\O))}}
%\mbox{$\int_0^t$} (\|f(u(0))\|_{L^4(\Om;L^2(\O))}
 + C\|u(s)\|_{L^p(\Om;L^2(\O))}) \d s+Ct^{\frac{\alpha}{2}} \\
\le&\ 
C(1+t^{\frac{\beta}{2}})
+
\mbox{$\int_0^t$} C\|u(s)\|_{L^p(\Om;L^2(\O))} \d s .
\end{align*}
Then applying Gronwall's inequality (see \cite[Lemma 7.1.1]{Daniel-1981}) yields \eqref{u-L2-estimate}.  

By applying the projection operator $\Pi_j$ to \eqref{mild-solution} and considering the result in the $L^p(\Omega;\dot{H}^{\alpha}(\mathcal{O}))$ norm, we have  
\begin{align*}
&\|\Pi_ju(t)\|_{L^p(\Omega;\dot{H}^{\alpha}(\mathcal{O}))} 
=
\|A^{\frac{\alpha}{2}}\Pi_ju(t)\|_{L^p(\Omega;L^2(\mathcal{O}))}\notag\\
&\le%1
\|A^{\frac{\alpha-\beta}{2}} e^{-tA}A^{\frac{\beta}{2}}\Pi_ju^0\|
_{L^p(\Omega;L^2(\mathcal{O}))}
\notag\\
&\quad+
\Big\|\int_0^{t}A^{\frac{\alpha}{2}}e^{-(t-s)A}\Pi_jf(u(s)) \d s\Big\|_{L^p(\Omega;L^2(\mathcal{O}))}
+
\Big\|\Pi_j\int_0^{t}A^{\frac{\alpha}{2}}e^{-(t-s)A}\d W(s)\Big\|_{L^p(\Omega;L^2(\mathcal{O}))}\notag\\
&\le%2
Ct^{-\frac{\alpha-\beta}{2}}\|A^{\frac{\beta}{2}}\Pi_ju^0\|_{ L^p(\Omega;L^2(\mathcal{O}))}\\
&\quad
+
C\int_0^{t}(t-s)^{-\frac{\alpha}{2}}\|f(u(s))\|_{L^p(\Omega;L^2(\mathcal{O}))} \d s 
+
\Big\|\Pi_j\int_0^{t}e^{-(t-s)A}\d W(s)\Big\|_{L^p(\Omega;\dot H^{\alpha}(\mathcal{O}))} \\
&\le%3
Ct^{-\frac{\alpha-\beta}{2}}\|u^0\|_{ L^p(\Omega;\dot H^{\beta}(\mathcal{O}))}\\
&\quad
+
C\int_0^{t}(t-s)^{-\frac{\alpha}{2}} (1+s^{\frac{\beta}{2}}) \d s 
+
\Big\|\Pi_j\int_0^{t}e^{-(t-s)A}\d W(s)\Big\|_{L^p(\Omega;\dot H^{\alpha}(\mathcal{O}))} \\
& \quad\mbox{(here we use the Lipschitz continuity of $f$ and \eqref{u-L2-estimate}, which is already proved)} \\
&\le
C t^{-\frac{\alpha-\beta}{2}} + C(t^{1-\frac{\alpha}{2}}+t^{1-\frac{\alpha-\beta}{2}}) + C, 
\end{align*}
where the last inequality uses { assumptions \eqref{t-alpha}--\eqref{mu-C}}. 
Then, by taking maximum in the above inequality among all $j\ge 1$, we obtain \eqref{B-Inf-estimate}. 

Next we prove that $u\in  C\big((\varepsilon,T]; L^p(\Om;L^2(\O))\big)$. Obviously,
by \eqref{Noise-Expansion} for $0<t_2<t_1\le T$ there hold
\begin{align}
%\begin{aligned}
&\Big\|\int_0^{t_2} (e^{-(t_1-s)A}-e^{-(t_2-s)A})\d W(s)\Big\|_{L^p(\Om;L^2(\O))}^2\notag\\
%&\sim
%\Big\|\int_0^{t_2} (e^{-(t_1-s)A}-e^{-(t_2-s)A})\d W(s)\Big\|_{L^2(\Om;L^2(\O))}^2
%\\
&\le
C\int_0^{t_2} \sum_{k=1}^{\infty}\mu_k\|(e^{-(t_1-s)A}-e^{-(t_2-s)A})\phi_k\|_{L^2(\O)}^2\d s\notag\\
%&\le 
%C\sum_{k=1}^{\infty}\mu_k\int_0^{t_2} (e^{-t_1\lambda_k}-e^{-t_2\lambda_k})^2e^{2s\lambda_k}\d s\notag\\
&\le
C\sum_{k=1}^{\infty}\mu_k(e^{-t_1\lambda_k}-e^{-t_2\lambda_k})^2\frac{e^{2t_2\lambda_k}-1}{2\lambda_k} \label{e-m-e-e}\\
%&\le
%C\sum_{k=1}^{\infty}\frac{\mu_k}{\lambda_k}\Big[(e^{-(t_1-t_2)\lambda_k}-1)^2-e^{-2t_2\lambda_k}(e^{-(t_1-t_2)\lambda_k}-1)^2\Big]\notag\\
%&\le
%C\sum_{k=1}^{\infty}\frac{\mu_k}{\lambda_k}(1-e^{-(t_1-t_2)\lambda_k})\notag\\
&\le{
C\sum_{k=1}^{\infty}\frac{\mu_k}{\lambda_k}(1-e^{-2(t_1-t_2)\lambda_k})}\notag\\
&\le 
C \big\| \mbox{$\int_0^{t_1-t_2}$} e^{-(t_1-t_2-s)A}\d W(s) \big\|_{L^2(\Omega;L^2(\mathcal{O}))}^2\notag\\
&\le 
C(t_1-t_2)^{\alpha}\notag
%\end{aligned}
\end{align}
and 
\be\label{noise-equiv}
\begin{aligned}
\big\|\mbox{$\int_{t_2}^{t_1}$} e^{-(t_1-s)A}\d W(s)\big\|_{L^p(\Om;L^2(\O))}^2
%&\le
%C\int_{t_2}^{t_1}\sum_{k=1}^{\infty}\mu_k\|e^{-(t_1-s)A}\phi_k\|_{L^2(\O)}^2\d s\\
&\le C 
\mbox{$\int_0^{t_1-t_2}$}\sum_{k=1}^{\infty}\mu_k\|e^{-(t_1-t_2-\sigma)A}\phi_k\|_{L^2(\O)}^2\d \sigma\\
&\sim
\big\|\mbox{$\int_{0}^{t_1-t_2}$}e^{-(t_1-t_2-s)A}\d W(s)\big\|_{L^2(\Om;L^2(\O))}^2\\
&\sim
(t_1-t_2)^{\alpha} , 
\end{aligned}
\en
where the last inequality is due to item (3) in Assumption \ref{thm-assumption}. 
%Assumption (3). 
Combining these estimates with \eqref{u-L2-estimate},
%{\color{red}Shall we change the order of proof so that \eqref{u-L2-estimate} could be already proved when we use it here?} 
we derive for $0<\varepsilon\le t_2<t_1\le T$ and $0<\delta<1$ that 
\begin{align*}
&\|u(t_1)-u(t_2)\|_{L^p(\Om;L^2(\O))}\\
&\le%1
\|e^{-t_1A}u^0-e^{-t_2A}u^0\|_{L^p(\Om;L^2(\O))}\\
&\quad+
\big\|\mbox{$\int_0^{t_2}$}(e^{-(t_1-s)A}-e^{-(t_2-s)A})f(u(s))\d s
+
\mbox{$\int_{t_2}^{t_1}$}e^{-(t_1-s)A}f(u(s))\d s\big\|_{L^p(\Om;L^2(\O))}\\
&\quad+
\big\|\mbox{$\int_0^{t_2}$} (e^{-(t_1-s)A}-e^{-(t_2-s)A})\d W(s)+\mbox{$\int_{t_2}^{t_1}$} e^{-(t_1-s)A}\d W(s)\big\|_{L^p(\Om;L^2(\O))}\\
&\le%2
\|A^{-\frac{\beta}{2}+\delta}e^{-t_2A}A^{-\delta}(e^{-(t_1-t_2)A}-I)A^{\frac{\beta}{2}}u^0\|_{L^p(\Om;L^2(\O))} \\
&\quad+
\mbox{$\int_0^{t_2}$} \|A^{\delta}e^{-(t_2-s)A}A^{-\delta}(e^{-(t_1-t_2)A}-I)f(u(s))\|_{L^p(\Om;L^2(\O))}\d s \\
&\quad+
C{\mbox{$\int_{t_2}^{t_1}$} (1+s^{\frac{\beta}{2}})\d s}\\
%\|u^0\|_{\b L^p(\Omega;\dot{H}^{\beta}(\O))}\d s
%\\& 
%\qquad\mbox{(here we use \eqref{u-L2-estimate})}} \\
&\quad+
\big\|\mbox{$\int_0^{t_2}$} (e^{-(t_1-s)A}-e^{-(t_2-s)A})\d W(s)\big\|_{L^p(\Om;L^2(\O))}+\big\|\mbox{$\int_{t_2}^{t_1}$} e^{-(t_1-s)A}\d W(s)\big\|_{L^p(\Om;L^2(\O))}\\
&\le
C{(t_2^{\frac{\beta}{2}-\delta}+1)}(t_1-t_2)^{\delta}\|u^0\|_{ L^p(\Omega;\dot{H}^{\beta}(\O))}%\mbox{\quad\color{red} In case that $\frac{\beta}{2}-\delta>0$}
\\
&\quad +
\mbox{$\int_0^{t_2}$} (t_2-s)^{-\delta}(t_1-t_2)^{\delta}{(1+s^{\frac{\beta}{2}})} \d s
%\|u^0\|_{\dot{H}^{\beta}(\O)}
 + C(t_1-t_2 + t_1^{1+\frac{\beta}{2}}-t_2^{1+\frac{\beta}{2}})\\
&\quad+
C(t_1-t_2)^{\frac{\alpha}{2}}\\
%&\le
%C{(\varepsilon^{\frac{\beta}{2}-\delta}+1)}(t_1-t_2)^{\delta}\|u^0\|_{\dot{H}^{\beta}(\O)} \\
%&\quad 
%+
%C(t_1-t_2)^{\delta}t_2^{1+\frac{\min(0,\beta)}{2}-\delta}\|u^0\|_{\dot{H}^{\beta}(\O)}\int_0^{1}(1-\xi)^{-\delta}\xi^{\frac{\beta}{2}}\d \xi 
%%\mbox{\color{red}\quad I got this, but it seems not really clear.}
%\\
%&\quad 
%+
%C(t_1-t_2)^{1+\frac{\min(0,\beta)}{2}} 
%+
%C(t_1-t_2)^{\frac{\alpha}{2}}\\
&\le
C{(\varepsilon^{\frac{\beta}{2}-\delta}+1)}(t_1-t_2)^{\delta} 
+ 
C(t_1-t_2)^{1+\frac{\min\{0,\beta\}}{2}} %原来这里写的是小括号\frac{\min(0,\beta)}{2}} 
+C(t_1-t_2)^{\frac{\alpha}{2}}. 
\end{align*} 
This means $u\in  C^{\delta}\big([\varepsilon,T]; L^p(\Om;L^2(\O))\big)$ for $\delta\in\big(0,{\min\{1+\frac{\min\{0,\beta\}}{2},\frac{\alpha}{2}\}}\big)$.
%{\color{red} Why can't be like $\delta\in (\cdot,\cdot]$?} 
The last two terms in the inequality above indicate that the second and third terms in expression \eqref{mild-solution} are in $C([0,T]; L^2(\Omega; L^2(\O)))$.
{Provided $\bar{\beta}=\min\{0,\beta\}$,} 
the first term in expression \eqref{mild-solution} is clearly in 
$C([0,T];L^2(\Omega;\dot H^{\bar{\beta}}(\O)))$
%with $\tilde{\beta}$
 because $e^{-tA}$ is a { strongly continuous} %$C^0$
semigroup on $\dot H^{\bar{\beta}}(\O)$. As a result, the mild solution $u$ is in $C([0,T];L^2(\Omega;\dot H^{\bar{\beta}}(\O)))$. 
This completes the proof of Proposition \ref{Regularity}.  

%{\color{red} What if $\beta>0$? Do we still have $C([0,T];L^2(\Omega;\dot H^{\beta}(\O)))$?}
%
%{\b I think no. Something is wrong with the noise term.}
%

\endproof

\bibliographystyle{abbrv}
\bibliography{Gui_Bibtex_SPDE}

\begin{thebibliography}{10}

\bibitem{Anton-Cohen-Larsson-Wang-2016}
R.~Anton, D.~Cohen, S.~Larsson, and X.~Wang.
\newblock Full discretization of semilinear stochastic wave equations driven by
  multiplicative noise.
\newblock {\em SIAM J. Numer. Anal.}, 54(2):1093--1119, 2016.

\bibitem{Anton-Cohen-2018}
R.~Anton, D.~Cohen, and L.~Quer-Sardanyons.
\newblock A fully discrete approximation of the one-dimensional stochastic heat
  equation.
\newblock {\em IMA J. Numer. Anal.}, 40(1):247--284, 2020.

\bibitem{Banjai-Lord-Molla-2021}
L.~Banjai, G.~Lord, and J.~Molla.
\newblock Strong convergence of a {V}erlet integrator for the semilinear
  stochastic wave equation.
\newblock {\em SIAM J. Numer. Anal.}, 59(4):1976--2003, 2021.

\bibitem{Bennett-1988}
C.~Bennett and R.~Sharpley.
\newblock {\em Interpolation of operators}, volume 129 of {\em Pure and Applied
  Mathematics}.
\newblock Academic Press, Inc., Boston, MA, 1988.

\bibitem{Bergh1976}
J.~Bergh and J.~L\"{o}fstr\"{o}m.
\newblock {\em Interpolation spaces. {A}n introduction}.
\newblock Springer-Verlag, Berlin-New York, 1976.
\newblock Grundlehren der Mathematischen Wissenschaften, No. 223.

\bibitem{Brehier-Cui-Hong-2018}
C.-E. Br\'{e}hier, J.~Cui, and J.~Hong.
\newblock Strong convergence rates of semidiscrete splitting approximations for
  the stochastic {A}llen-{C}ahn equation.
\newblock {\em IMA J. Numer. Anal.}, 39(4):2096--2134, 2019.

\bibitem{Cao-2018}
Y.~Cao, J.~Hong, and Z.~Liu.
\newblock Finite element approximations for second-order stochastic
  differential equation driven by fractional {B}rownian motion.
\newblock {\em IMA J. Numer. Anal.}, 38(1):184--197, 2018.

\bibitem{Cao-2007}
Y.~Cao and L.~Yin.
\newblock Spectral {G}alerkin method for stochastic wave equations driven by
  space-time white noise.
\newblock {\em Commun. Pure Appl. Anal.}, 6(3):607--617, 2007.

\bibitem{A-Lang-2021}
D.~Cohen and A.~Lang.
\newblock Numerical approximation and simulation of the stochastic wave
  equation on the sphere.
\newblock {\em Calcolo}, 59(3):Paper No. 32, 32, 2022.

\bibitem{Da-Prato1992}
G.~Da~Prato and J.~Zabczyk.
\newblock {\em Stochastic equations in infinite dimensions}, volume~44 of {\em
  Encyclopedia of Mathematics and its Applications}.
\newblock Cambridge University Press, Cambridge, 1992.

\bibitem{Evans}
L.~C. Evans.
\newblock {\em Partial differential equations}, volume~19 of {\em Graduate
  Studies in Mathematics}.
\newblock American Mathematical Society, Providence, RI, second edition, 2010.

\bibitem{Gyongy-1998}
I.~Gy\"{o}ngy.
\newblock Lattice approximations for stochastic quasi-linear parabolic partial
  differential equations driven by space-time white noise. {I}.
\newblock {\em Potential Anal.}, 9(1):1--25, 1998.

\bibitem{Gyongy-1999}
I.~Gy\"{o}ngy.
\newblock Lattice approximations for stochastic quasi-linear parabolic partial
  differential equations driven by space-time white noise. {II}.
\newblock {\em Potential Anal.}, 11(1):1--37, 1999.

\bibitem{Daniel-1981}
D.~Henry.
\newblock {\em Geometric theory of semilinear parabolic equations}, volume 840
  of {\em Lecture Notes in Mathematics}.
\newblock Springer-Verlag, Berlin-New York, 1981.

\bibitem{Higham-2021}
D.~J. Higham and P.~E. Kloeden.
\newblock {\em An introduction to the numerical simulation of stochastic
  differential equations}.
\newblock Society for Industrial and Applied Mathematics (SIAM), Philadelphia,
  PA, [2021] \copyright 2021.

\bibitem{Higham-2002}
D.~J. Higham, X.~Mao, and A.~M. Stuart.
\newblock Strong convergence of {E}uler-type methods for nonlinear stochastic
  differential equations.
\newblock {\em SIAM J. Numer. Anal.}, 40(3):1041--1063, 2002.

\bibitem{Hutzenthaler-2011}
M.~Hutzenthaler and A.~Jentzen.
\newblock Convergence of the stochastic {E}uler scheme for locally {L}ipschitz
  coefficients.
\newblock {\em Found. Comput. Math.}, 11(6):657--706, 2011.

\bibitem{Hutzenthaler-2020}
M.~Hutzenthaler and A.~Jentzen.
\newblock On a perturbation theory and on strong convergence rates for
  stochastic ordinary and partial differential equations with nonglobally
  monotone coefficients.
\newblock {\em Ann. Probab.}, 48(1):53--93, 2020.

\bibitem{Hutzenthaler-2012}
M.~Hutzenthaler, A.~Jentzen, and P.~E. Kloeden.
\newblock Strong convergence of an explicit numerical method for {SDE}s with
  nonglobally {L}ipschitz continuous coefficients.
\newblock {\em Ann. Appl. Probab.}, 22(4):1611--1641, 2012.

\bibitem{Jentzen-Kloeden-2009}
A.~Jentzen and P.~E. Kloeden.
\newblock Overcoming the order barrier in the numerical approximation of
  stochastic partial differential equations with additive space-time noise.
\newblock {\em Proc. R. Soc. Lond. Ser. A Math. Phys. Eng. Sci.},
  465(2102):649--667, 2009.

\bibitem{Kovas-Larsson-Saedpanah-2010}
M.~Kov\'{a}cs, S.~Larsson, and F.~Saedpanah.
\newblock Finite element approximation of the linear stochastic wave equation
  with additive noise.
\newblock {\em SIAM J. Numer. Anal.}, 48(2):408--427, 2010.

\bibitem{kress-2014}
R.~Kress.
\newblock {\em Linear integral equations}, volume~82 of {\em Applied
  Mathematical Sciences}.
\newblock Springer, New York, third edition, 2014.

\bibitem{Kruse-2014}
R.~Kruse.
\newblock {\em Strong and weak approximation of semilinear stochastic evolution
  equations}, volume 2093 of {\em Lecture Notes in Mathematics}.
\newblock Springer, Cham, 2014.

\bibitem{A-Lang-2017}
A.~Lang, A.~Petersson, and A.~Thalhammer.
\newblock Mean-square stability analysis of approximations of stochastic
  differential equations in infinite dimensions.
\newblock {\em BIT}, 57(4):963--990, 2017.

\bibitem{Li-Ma-2020}
B.~Li and S.~Ma.
\newblock A high-order exponential integrator for nonlinear parabolic equations
  with nonsmooth initial data.
\newblock {\em J. Sci. Comput.}, 87(1):Paper No. 23, 16, 2021.

\bibitem{Catherine-2014}
G.~J. Lord, C.~E. Powell, and T.~Shardlow.
\newblock {\em An introduction to computational stochastic {PDE}s}.
\newblock Cambridge Texts in Applied Mathematics. Cambridge University Press,
  New York, 2014.

\bibitem{Malliavin-1995}
P.~Malliavin.
\newblock {\em Integration and probability}, volume 157 of {\em Graduate Texts
  in Mathematics}.
\newblock Springer-Verlag, New York, 1995.
\newblock With the collaboration of H\'{e}l\`ene Airault, Leslie Kay and
  G\'{e}rard Letac, Edited and translated from the French by Kay, With a
  foreword by Mark Pinsky.

\bibitem{Mclean-2000}
W.~McLean.
\newblock {\em Strongly elliptic systems and boundary integral equations}.
\newblock Cambridge University Press, Cambridge, 2000.

\bibitem{Mukam-Tambue-2020}
J.~D. Mukam and A.~Tambue.
\newblock Strong convergence of a stochastic {R}osenbrock-type scheme for the
  finite element discretization of semilinear {SPDE}s driven by multiplicative
  and additive noise.
\newblock {\em Stochastic Process. Appl.}, 130(8):4968--5005, 2020.

\bibitem{Mukam-Tambue-2020-2}
J.~D. Mukam and A.~Tambue.
\newblock Strong convergence of the linear implicit {E}uler method for the
  finite element discretization of semilinear non-autonomous {SPDE}s driven by
  multiplicative or additive noise.
\newblock {\em Appl. Numer. Math.}, 147:222--253, 2020.

\bibitem{Rockner-2007}
C.~Pr\'{e}v\^{o}t and M.~R\"{o}ckner.
\newblock {\em A concise course on stochastic partial differential equations},
  volume 1905 of {\em Lecture Notes in Mathematics}.
\newblock Springer, Berlin, 2007.

\bibitem{Qi-Wang-2019}
R.~Qi and X.~Wang.
\newblock Error estimates of finite element method for semilinear stochastic
  strongly damped wave equation.
\newblock {\em IMA J. Numer. Anal.}, 39(3):1594--1626, 2019.

\bibitem{Stein-functional}
E.~M. Stein and R.~Shakarchi.
\newblock {\em Functional analysis}, volume~4 of {\em Princeton Lectures in
  Analysis}.
\newblock Princeton University Press, Princeton, NJ, 2011.
\newblock Introduction to further topics in analysis.

\bibitem{Wang-2015-2}
X.~Wang.
\newblock Weak error estimates of the exponential {E}uler scheme for
  semi-linear {SPDE}s without {M}alliavin calculus.
\newblock {\em Discrete Contin. Dyn. Syst.}, 36(1):481--497, 2016.

\bibitem{Wang-2017}
X.~Wang.
\newblock Strong convergence rates of the linear implicit {E}uler method for
  the finite element discretization of {SPDE}s with additive noise.
\newblock {\em IMA J. Numer. Anal.}, 37(2):965--984, 2017.

\bibitem{Wang-2020}
X.~Wang.
\newblock An efficient explicit full-discrete scheme for strong approximation
  of stochastic {A}llen-{C}ahn equation.
\newblock {\em Stochastic Process. Appl.}, 130(10):6271--6299, 2020.

\bibitem{Wang-Gan-Tang-2014}
X.~Wang, S.~Gan, and J.~Tang.
\newblock Higher order strong approximations of semilinear stochastic wave
  equation with additive space-time white noise.
\newblock {\em SIAM J. Sci. Comput.}, 36(6):A2611--A2632, 2014.

\bibitem{Wang-2015}
X.~Wang and R.~Qi.
\newblock A note on an accelerated exponential {E}uler method for parabolic
  {SPDE}s with additive noise.
\newblock {\em Appl. Math. Lett.}, 46:31--37, 2015.

\bibitem{Yan-2005}
Y.~Yan.
\newblock Galerkin finite element methods for stochastic parabolic partial
  differential equations.
\newblock {\em SIAM J. Numer. Anal.}, 43(4):1363--1384, 2005.

\end{thebibliography}
\end{document}